\documentclass[a4paper, 11pt]{amsart}
\usepackage{graphicx}
\usepackage{amssymb}
\usepackage[margin=3cm]{geometry}
\usepackage{color}
\usepackage{enumerate}
\usepackage[T1]{fontenc}
\usepackage[utf8]{inputenc}
\usepackage{mathtools}
\usepackage{comment}
\usepackage{float}
\usepackage{mathrsfs}
\newtheorem{thm}{Theorem}[section]
\newtheorem{cor}[thm]{Corollary}
\newtheorem{lem}[thm]{Lemma}

\newtheorem{cex}[thm]{Counterexample}
\theoremstyle{definition}
\newtheorem{defn}[thm]{Definition}
\theoremstyle{remark}
\newtheorem{rem}[thm]{Remark}

\numberwithin{equation}{section}
\begin{document}
\title[]{On the explicit representation of the trace space $H^{\frac{3}{2}}$ and of the solutions to biharmonic Dirichlet problems on Lipschitz domains via multi-parameter Steklov problems}

%\title[]{Multi-parameter eigenvalue problems of Steklov type for the biharmonic operator and Fourier characterization of the trace spaces of $H^2(\Omega)$ on Lipschitz domains}
%Multi-parameter Steklov problems for the biharmonic operator: asymptotics}%
\author{Pier Domenico Lamberti}%
\address{Universit\`a degli Studi di Padova, Dipartimento di Matematica ``Tullio Levi-Civita'', Via Trieste 63, 35121 Padova, Italy}%
\email{lamberti@math.unipd.it}%
\author{Luigi Provenzano}%
\address{Universit\`a degli Studi di Padova, Dipartimento di Matematica ``Tullio Levi-Civita'', Via Trieste 63, 35121 Padova, Italy}%
\email{luigi.provenzano@math.unipd.it}%
\thanks{The authors are very thankful to Professors Giles Auchmuty and Victor I. Burenkov for useful discussions and references. This work was partially supported by the research project ``Progetto GNAMPA 2019 - Analisi spettrale per operatori ellittici con condizioni di Steklov o parzialmente incernierate''. The authors are members of the Gruppo Nazionale per l'Analisi Matematica, la Probabilit\`a e le loro Applicazioni (GNAMPA) of the Istituto Nazionale di Alta Matematica (INdAM)}%
\date{\today}%
%\date{February, 22, 2018}%
%\dedicatory{}%
%\commby{}%
% ----------------------------------------------------------------
\begin{abstract}
We consider the problem of describing the traces of functions in $H^2(\Omega)$ on the boundary of a Lipschitz domain $\Omega$ of $\mathbb R^N$, $N\geq 2$. We provide a definition of those spaces, in particular of $H^{\frac{3}{2}}(\partial\Omega)$, by means of Fourier series associated with the eigenfunctions of new multi-parameter biharmonic Steklov problems which we introduce with this specific purpose. These definitions coincide with the classical ones when the domain is smooth. Our spaces allow to represent in series the solutions to the biharmonic Dirichlet problem. Moreover, a few spectral properties of the multi-parameter biharmonic Steklov problems are considered, as well as explicit examples. Our approach is similar to that developed by G. Auchmuty for the space $H^1(\Omega)$, based on the classical second order Steklov problem.
%We introduce an eigenvalue problem for the biharmonic operator subject to Neumann boundary conditions on a compact Riemannian manifold with boundary.
\end{abstract}
\maketitle
%\tableofcontents
% ----------------------------------------------------------------

%%%%%%%%%%%%%%%%%%%%%%%%%%%%%%%%%%%%%%%%%%%%%%%%%%%%%%%%%%%%%%%%%%%%%%%%%%%%%%%%%%%%%%%%%%%%%%%%%%%%%%%%%%%%%%%%%%%%%%%%%%%%%%%%%%%%%%%%
%%%%%%%%%%%%%%%%%%%%%%%%%%%%%%%%%%%%%%%%%%%%%%%%%%%%%%%%%%%%%%%%%%%%%%%%%%%%%%%%%%%%%%%%%%%%%%%%%%%%%%%%%%%%%%%%%%%%%%%%%%%%%%%%%%%%%%%%
%%%%%%%%%%%%%%%%%%%%%%%%%%%%%%%%%%%%%%%%%%%%%%%%%%%%%%%%%%%%%%%%%%%%%%%%%%%%%%%%%%%%%%%%%%%%%%%%%%%%%%%%%%%%%%%%%%%%%%%%%%%%%%%%%%%%%%%%

%%%%%%%%%%%%%%%%%%%%%%%%%%%%%%%%%%%%%%%%%%%%                INTRODUCTION                  %%%%%%%%%%%%%%%%%%%%%%%%%%%%%%%%%%%%%%%%%%%%% 

%%%%%%%%%%%%%%%%%%%%%%%%%%%%%%%%%%%%%%%%%%%%%%%%%%%%%%%%%%%%%%%%%%%%%%%%%%%%%%%%%%%%%%%%%%%%%%%%%%%%%%%%%%%%%%%%%%%%%%%%%%%%%%%%%%%%%%%%
%%%%%%%%%%%%%%%%%%%%%%%%%%%%%%%%%%%%%%%%%%%%%%%%%%%%%%%%%%%%%%%%%%%%%%%%%%%%%%%%%%%%%%%%%%%%%%%%%%%%%%%%%%%%%%%%%%%%%%%%%%%%%%%%%%%%%%%%
%%%%%%%%%%%%%%%%%%%%%%%%%%%%%%%%%%%%%%%%%%%%%%%%%%%%%%%%%%%%%%%%%%%%%%%%%%%%%%%%%%%%%%%%%%%%%%%%%%%%%%%%%%%%%%%%%%%%%%%%%%%%%%%%%%%%%%%%

\section{Introduction}
We consider the trace spaces of functions in $H^2(\Omega)$ when $\Omega$ is a bounded Lipschitz domain in $\mathbb R^N$, briefly $\Omega$ is of class $C^{0,1}$, for $N\geq 2$. It is well known that there exists a linear and continuous operator $\Gamma$ called the {\it total trace}, from $H^2(\Omega)$ to $L^2(\partial\Omega)\times L^2(\partial\Omega)$ defined by $\Gamma(u)=\left(\gamma_0(u),\gamma_1(u)\right)$, where $\gamma_0(u)$ is the {\it trace} of $u$ on $\partial\Omega$ and $\gamma_1(u)$ is the {\it normal derivative} of $u$. In particular, for $u\in C^{2}(\overline\Omega)$, $\gamma_0(u)=u_{|_{\partial\Omega}}$ and $\gamma_1(u)=\frac{\partial u}{\partial\nu}=\nabla u_{|_{\partial\Omega}}\cdot\nu$, where $\nu$ denotes the outer unit normal to $\partial\Omega$.

A relevant problem in the theory of Sobolev Spaces consists in describing the  {\it trace spaces} $\gamma_0(H^2(\Omega))$, $\gamma_1(H^2(\Omega))$, and the {\it total trace space} $\Gamma(H^2(\Omega))$. This problem has important implications in the study of solutions to fourth order elliptic partial differential equations. 

From a historical point of view, this issue finds its origins in \cite{hadamard} where 
J. Hadamard proposed his famous counterexample pointing out the importance  to understand  which conditions on the datum $g$ guarantee that the solution $v$ to the Dirichlet problem 
$$
\left\{ 
\begin{array}{ll}\Delta v=0,& {\rm in}\ \Omega,\\
v=g,& {\rm on }\ \partial \Omega,
\end{array}
\right.
$$ 
has square summable gradient.  
% This problem was solved in the 50's starting with  the work \cite{miranda} of C. Miranda 
%who proved that if $g$ is H\"{o}lder continuous with exponent $1-1/p$ then $\nabla v\in L^p(\Omega)$. 
 In modern terms,  this  problem  can be reformulated as the problem of finding necessary and sufficient conditions on $g$ such that $g=\gamma_0(u)$ for some $u\in H^1(\Omega)$.

 If the domain $\Omega$ is of class $C^{2,1}$, then it is known that $\gamma_0(H^2(\Omega))=H^{\frac{3}{2}}(\partial\Omega)$, $\gamma_1(H^2(\Omega))=H^{\frac{1}{2}}(\partial\Omega)$, and $\Gamma(H^2(\Omega))=H^{\frac{3}{2}}(\partial\Omega)\times H^{\frac{1}{2}}(\partial\Omega)$, where $H^{\frac{3}{2}}(\partial\Omega)$ and $H^{\frac{1}{2}}(\partial\Omega)$ are the classical Sobolev spaces of fractional order (see e.g., \cite{grisvard,necas} for their definitions). However, if $\Omega$ is an arbitrary bounded domain of class $C^{0,1}$ there is no such a simple description and not many results are available in the literature.

We note that a complete description of the traces of all derivatives up to the order $m-1$ of a function $u\in H^m(\Omega)$ is due to O. Besov who provided an explicit but quite technical representation theorem, see \cite{besov19721,besov19722}, see also \cite{besov}.  Simpler descriptions are not available with the exception of a few special cases. For example, when $\Omega$ is a polygon in $\mathbb R^2$ the trace spaces are described by using the classical trace theorem applied to each side of the polygon, complemented with suitable compatibility conditions at the vertexes, see \cite{grisvard} also for higher dimensional polyhedra. For more general planar domains another simple description is given in \cite{geymonat}.

Our list  of references cannot be exhaustive and we refer to the recent monograph  \cite{mazya_polyhedral} which  treats the trace problem in presence of corner or conical singularities in $\mathbb R^3$, as well as further results on $N$-dimensional polyhedra. We also quote  the fundamental paper ~\cite{kon} by  V.~Kondrat'ev for a pioneering work in this type of problems.

Thus, the definition of the space $H^{\frac{3}{2}}(\partial\Omega)$ turns out to be problematic and for this reason sometimes the space $H^{\frac{3}{2}}(\partial\Omega)$ is simply defined by setting $H^{\frac{3}{2}}(\partial\Omega):=\gamma_0(H^2(\Omega))$ without providing an explicit representation. Note that standard definitions of $H^s(\partial\Omega)$ when $s\in(1,2]$ require that $\Omega$ is of class at least $C^{2}$.

In the present paper we provide decompositions of the space $H^2(\Omega)$ of the form $H^2(\Omega)=H^2_{\mu,D}(\Omega)+\mathcal{H}^2_{0,N}(\Omega)$ and $H^2(\Omega)=H^2_{\lambda,N}(\Omega)+\mathcal{H}^2_{0,D}(\Omega)$. The spaces $\mathcal{H}^2_{0,N}(\Omega)$ and $\mathcal{H}^2_{0,D}(\Omega)$ are the subspaces of $H^2(\Omega)$ of those functions $u$ such that $\gamma_1(u)=0$ and $\gamma_0(u)=0$, respectively. The spaces $H^2_{\mu,D}(\Omega)$ and $H^2_{\lambda,N}(\Omega)$ are associated with suitable Steklov problems of biharmonic type (namely, problems \eqref{BSM} and \eqref{BSL} described here below), depending on real parameters $\mu,\lambda$, and admit Fourier bases of Steklov eigenfunctions, see \eqref{HM} and \eqref{HL}. Under the sole assumptions that $\Omega$ is of class $C^{0,1}$, we use those bases to define in a natural way two spaces at the boundary which we denote by $\mathcal S^{\frac{3}{2}}(\partial\Omega)$ and $\mathcal S^{\frac{1}{2}}(\partial\Omega)$ and we prove that
$$
\gamma_0(H^2(\Omega))=\gamma_0(H^2_{\lambda,N}(\Omega))=\mathcal S^{\frac{3}{2}}(\partial\Omega)
$$
and
$$
\gamma_1(H^2(\Omega))=\gamma_1(H^2_{\mu,D}(\Omega))=\mathcal S^{\frac{1}{2}}(\partial\Omega),
$$
%we are able to describe $\gamma_0(H^2(\Omega))$ and $\gamma_1(H^2(\Omega))$ in terms of the spaces $\mathcal S_{\lambda}^{\frac{3}{2}}(\partial\Omega)=\gamma_0(H^2_{\lambda,N}(\Omega))$ and $\mathcal S_{\mu}^{\frac{1}{2}}(\partial\Omega)=\gamma_1(H^2_{\mu,D}(\Omega))$. 
%These spaces have an explicit and simple characterization:  all functions belonging to  them can be written as a Fourier series in terms of biharmonic Steklov eigenfunctions whose coefficients satisfy certain summability properties (
see Theorem \ref{traces}. Thus, if one would like to define the space $H^\frac{3}{2}(\partial\Omega)$ as $\gamma_0(H^2(\Omega))$, our result gives an explicit description of  $H^\frac{3}{2}(\partial\Omega)$.

%It will be also clear that such problems are the appropriate problems to describe the trace spaces $\gamma_0(H^2(\Omega))$ and $\gamma_1(H^2(\Omega))$ when $\Omega$ is of class $C^{0,1}$.
It turns out that the analysis of problems \eqref{BSM}-\eqref{BSL} provides further information on the total trace $\Gamma(H^2(\Omega))$. In particular, we prove the inclusion $\Gamma(H^2(\Omega))\subseteq\mathcal S^{\frac{3}{2}}(\partial\Omega)\times\mathcal S^{\frac{1}{2}}(\partial\Omega)$ and show that in general this inclusion is strict if $\Omega$ is assumed to be only of class $C^{0,1}$. Moreover, we show that any couple $(f,g)\in \mathcal S^{\frac{3}{2}}(\partial\Omega)\times\mathcal S^{\frac{1}{2}}(\partial\Omega)$ belongs to $\Gamma(H^2(\Omega))$ if and only if it satisfies a certain compatibility condition, see Theorem \ref{total_trace_2}. %This result is somehow a natural extension , if we think, e.g., of \cite{geymonat}.

If $\Omega$ is of class $C^{2,1}$, we recover the classical result, namely $\Gamma(H^2(\Omega))=\mathcal S^{\frac{3}{2}}(\partial\Omega)\times\mathcal S^{\frac{1}{2}}(\partial\Omega)$, which implies that $\mathcal S^{\frac{3}{2}}(\partial\Omega)=H^{\frac{3}{2}}(\partial\Omega)$ and $\mathcal S^{\frac{1}{2}}(\partial\Omega)=H^{\frac{1}{2}}(\partial\Omega)$.
% In particular we provide alternative and simple descriptions of the spaces $H^{\frac{3}{2}}(\partial\Omega)$ and $H^{\frac{1}{2}}(\partial\Omega)$ in terms of Fourier expansions associated with biharmonic Steklov eigenfunctions. Actually, if we set $H^\frac{3}{2}(\partial\Omega):=\gamma_0(H^2(\Omega))$ by definition (when $\Omega$ is of class $C^{0,1}$), we have the same explicit simple representation in terms of Fourier series.

The two families of problems which we are going to introduce depend on a parameter $\sigma\in\big(-\frac{1}{N-1},1\big)$, which in applications to linear elasticity represents the Poisson coefficient of the elastic material of the underlying system for $N=2$.

The first family of ${\rm BS}_{\mu}$ - `Biharmonic Steklov $\mu$' problems is defined as follows:
\begin{equation}\tag{${\rm BS}_{\mu}$}\label{BSM}
\begin{cases}
\Delta^2v=0, & {\rm in\ }\Omega,\\
(1-\sigma)\frac{\partial^2 v}{\partial\nu^2}+\sigma\Delta v=\lambda(\mu)\frac{\partial v}{\partial\nu}, & {\rm on\ }\partial\Omega,\\
-(1-\sigma){\rm div}_{\partial\Omega}(D^2v\cdot\nu)_{\partial\Omega}-\frac{\partial \Delta v}{\partial\nu}=\mu v, & {\rm on\ }\partial\Omega,
\end{cases}
\end{equation}
in the unknowns $v,\lambda(\mu)$, where $\mu\in\mathbb R$ is fixed. Here $D^2 u$ denotes the Hessian matrix of $u$, ${\rm div}_{\partial\Omega}F:={\rm div}F-(\nabla F\cdot \nu)\nu$ denotes the tangential divergence of a vector field $F$ and $F_{\partial\Omega}:=F-(F\cdot\nu)\nu$ denotes the tangential component of $F$.% Problem \eqref{BSM} is new in the literature, up to our knowledge. 

The second family of ${\rm BS}_{\lambda}$ - `Biharmonic Steklov $\lambda$' problems is defined as follows:
\begin{equation}\tag{${\rm BS}_{\lambda}$}\label{BSL}
\begin{cases}
\Delta^2u=0, & {\rm in\ }\Omega,\\
(1-\sigma)\frac{\partial^2 u}{\partial\nu^2}+\sigma\Delta u=\lambda\frac{\partial u}{\partial\nu}, & {\rm on\ }\partial\Omega,\\
-(1-\sigma){\rm div}_{\partial\Omega}(D^2u\cdot\nu)_{\partial\Omega}-\frac{\partial \Delta u}{\partial\nu}=\mu(\lambda) u, & {\rm on\ }\partial\Omega,
\end{cases}
\end{equation}
in the unknowns $u,\mu(\lambda)$, where $\lambda\in\mathbb R$ is fixed. 

Note that since $\Omega$ is assumed to be of class $C^{0,1}$, problems \eqref{BSM} and \eqref{BSL} have to be considered in the weak sense, see \eqref{weak_BSM} and \eqref{weak_BSL} for the appropriate formulations.

Up to our knowledge, the Steklov problems \eqref{BSM} and \eqref{BSL} are new in the literature. Other Steklov-type problems for the biharmonic operator have been discussed in the literature. We mention the DBS - `Dirichlet Biharmonic Steklov' problem
\begin{equation}\tag{${\rm DBS}$}\label{DBS}
\begin{cases}
\Delta^2v=0, & {\rm in\ }\Omega,\\
(1-\sigma)\frac{\partial^2 v}{\partial\nu^2}+\sigma\Delta v=\eta\frac{\partial v}{\partial\nu}, & {\rm on\ }\partial\Omega,\\
v=0, & {\rm on\ }\partial\Omega,
\end{cases}
\end{equation}
in the unknowns $v,\eta$, and the NBS - `Neumann Biharmonic Steklov' problem
\begin{equation}\tag{${\rm NBS}$}\label{NBS}
\begin{cases}
\Delta^2u=0, & {\rm in\ }\Omega,\\
\frac{\partial u}{\partial\nu}=0, & {\rm on\ }\partial\Omega,\\
-(1-\sigma){\rm div}_{\partial\Omega}(D^2u\cdot\nu)_{\partial\Omega}-\frac{\partial \Delta u}{\partial\nu}=\xi u, & {\rm on\ }\partial\Omega,
\end{cases}
\end{equation}
in the unknowns $u,\xi$. Problem \eqref{DBS} for $\sigma=1$ has been studied by many authors (see e.g., \cite{auchmuty_svg,bucurgazzola,ferrerogazzola,ferrero_lamberti,fichera,kuttler_sigillito,liu_1}); for the case $\sigma\ne 1$ we refer to \cite{buoso_proc}, see also \cite{auch,ferrero_lamberti} for $\sigma=0$. Problem \eqref{NBS} has been discussed in \cite{kuttler_sigillito,liu_2,liu_1} for $\sigma=1$. We point out that problem \eqref{BSL} with $\sigma=\lambda=0$ has been introduced in \cite{buosoprovenzano} as the natural fourth order generalization of the classical Steklov problem for the Laplacian (see also \cite{bcp}). As we shall see, problem \eqref{BSL} shares much more analogies with the classical Steklov problem than those already presented in \cite{buosoprovenzano}, in particular it plays a role in describing the space $\gamma_0(H^2(\Omega))$ similar to that played by the Steklov problem for the Laplacian in describing $\gamma_0(H^1(\Omega))$ (cf. \cite{auchmuty_steklov}).

If $\mu<0$, problem \eqref{BSM} has a discrete spectrum which consists of a divergent sequence  $\left\{\lambda_j(\mu)\right\}_{j=1}^{\infty}$ of non-negative eigenvalues of finite multiplicity. Similarly, if $\lambda<\eta_1$, where $\eta_1>0$ is the first eigenvalue of \eqref{DBS}, problem \eqref{BSL} has a discrete spectrum which consists of a divergent sequence  $\left\{\mu_j(\lambda)\right\}_{j=1}^{\infty}$ of eigenvalues of finite multiplicity and  bounded from below. (For other values of $\mu$ and $\lambda$ the description of the spectra of \eqref{BSM} and \eqref{BSL} is more involved, see Appendix \ref{appendixA}.)

The eigenfunctions associated with the eigenvalues $\lambda_j(\mu)$ define a Hilbert basis of the above mentioned space $H^2_{\mu,D}(\Omega)$ which is the orthogonal complement in $H^2(\Omega)$ of $\mathcal H^2_{0,N}(\Omega)$ with respect to a suitable scalar product. Moreover, the normal derivatives of those eigenfunctions allow to define the above mentioned space $\mathcal S^{\frac{1}{2}}(\partial\Omega)$, see \eqref{S_mu}. Similarly, the eigenfunctions associated with the eigenvalues $\mu_j(\lambda)$ define a Hilbert basis of the space $H^2_{\lambda,N}(\Omega)$  which is the orthogonal complement in $H^2(\Omega)$ of $\mathcal H^2_{0,D}(\Omega)$ with respect to a suitable scalar product. Moreover, the traces of those eigenfunctions  allow to define the space $\mathcal S^{\frac{3}{2}}(\partial\Omega)$, see \eqref{S_lambda}.

The definitions in \eqref{S_lambda} and \eqref{S_mu} are given by means of Fourier series and the coefficients in such expansions need to satisfy certain summability conditions, which are strictly related to the asymptotic behavior of the eigenvalues of \eqref{BSL} and \eqref{BSM}. Note that 
%Associated to problems \eqref{BSM} and \eqref{BSL} we find two families of Hilbert basis of eigenfunctions of the spaces $H^2_{\mu,D}(\Omega)$ and $H^2_{\lambda,N}(\Omega)$, respectively (see \eqref{HM} and \eqref{HL} for definitions). These spaces are the orthogonal in $H^2(\Omega)$ (with respect to suitable equivalent scalar products) to the spaces $\mathcal H_{N,0}(\Omega)$ and $\mathcal H_{D,0}(\Omega)$, respectively. The spaces $\mathcal S_{\lambda}^{\frac{3}{2}}(\partial\Omega)$ and  $\mathcal S_{\mu}^{\frac{1}{2}}(\partial\Omega)$ are defined respectively by \eqref{S_lambda} and \eqref{S_mu} and are related to the trace spaces $\gamma_0(H^2_{\lambda,N})(\Omega)$ and $\gamma_1(H^2_{\mu,D})(\Omega)$. Their description is simple: all functions in such spaces can be written as a Fourier series in terms of the $L^2(\partial\Omega)$ normalized traces of eigenfunctions of \eqref{BSL} and \eqref{BSM}. The coefficients of such expansions need to satisfy a certain summability condition, which is strictly related to the asymptotic behavior of the eigenvalues of \eqref{BSL} and \eqref{BSM}. In fact, we note that (see Appendix \ref{asymptotic})
\begin{equation}\label{weyl}
\mu_j(\lambda)\sim C_N\left(\frac{j}{|\partial\Omega|}\right)^\frac{3}{N-1}{\rm\ \ \ and\ \ \ }\lambda_j(\mu)\sim C_N'\left(\frac{j}{|\partial\Omega|}\right)^\frac{1}{N-1}\,,{\rm\ as\ }j\rightarrow+\infty,
\end{equation}
where $C_N,C_N'$ depend only on $N$, see Appendix \ref{asymptotic}. In view of  \eqref{weyl} and \eqref{S_lambda}-\eqref{S_mu}, we can identify the space $\mathcal S^{\frac{3}{2}}(\partial\Omega)$ with the space of sequences
\begin{equation}\label{weyl1}
\left\{(s_j)_{j=1}^{\infty}\in\mathbb R^{\infty}:(j^{\frac{3}{2(N-1)}}s_j)_{j=1}^{\infty}\in l^2\right\}
\end{equation}
and the space $\mathcal S^{\frac{1}{2}}(\partial\Omega)$ with the space
\begin{equation}\label{weyl2}
\left\{(s_j)_{j=1}^{\infty}\in\mathbb R^{\infty}:(j^{\frac{1}{2(N-1)}}s_j)_{j=1}^{\infty}\in l^2\right\}.
\end{equation}
        Observe the natural  appearance of the exponents $\frac{3}{2}$ and $\frac{1}{2}$ in \eqref{weyl1} and \eqref{weyl2}. It is remarkable that, in essence, a summability condition analogous  to that in \eqref{weyl2} is already present in  \cite[Formula~(3)]{hadamard}  for the case of the unit disk $D$ of the plane and the space $H^{\frac{1}{2}}(\partial D)=\gamma_0(H^1(D))$.

Using the representations \eqref{S_lambda} and \eqref{S_mu} we are able to provide necessary and sufficient conditions for the solvability in $H^2(\Omega)$ of the Dirichlet problem
\begin{equation}\label{Dirichlet_problem}
\begin{cases}
\Delta^2 u=0\,, & {\rm in\ }\Omega,\\
u=f\,, & {\rm on\ }\partial\Omega,\\
\frac{\partial u}{\partial\nu}=g\,, & {\rm on\ }\partial\Omega,
\end{cases}
\end{equation}
under the sole assumption that $\Omega$ is of class $C^{0,1}$, and to represent in Fourier series the solutions. We note that different necessary and sufficient conditions for the solvability of problem \eqref{Dirichlet_problem} in the space $H(\Delta,\Omega)=\left\{u\in H^1(\Omega):\Delta u\in L^2(\Omega)\right\}$ have been found in \cite{auch} by using the \eqref{DBS} problem with $\sigma=1$ and the classical Dirichlet-to-Neumann map. We refer to \cite{barton,verchota_dir,verchota} for a different approach to the solvability of higher order problems on Lipschitz domains.

%Moreover, thanks to the explicit representations of the dual spaces $\mathcal S_{\lambda}^{-\frac{3}{2}}(\partial\Omega)$ and $\mathcal S_{\mu}^{-\frac{1}{2}}(\partial\Omega)$ we can represent in Fourier series the solutions to the Neumann problem

%\begin{equation}\label{Neumann_problem}
%\begin{cases}
%\Delta^2 w=0\,, & {\rm in\ }\Omega,\\
%(1-\sigma)\frac{\partial^2 w}{\partial\nu^2}+\sigma\Delta w=F, & {\rm on\ }\partial\Omega,\\
%-(1-\sigma){\rm div}_{\partial\Omega}(D^2w\cdot\nu)_{\partial\Omega}-\frac{\partial \Delta w}{\partial\nu}=G, & {\rm on\ }\partial\Omega,
%\end{cases}
%\end{equation}
%for $(F,G)\in\mathcal S_{\lambda}^{-\frac{3}{2}}(\partial\Omega)\times\mathcal S_{\mu}^{-\frac{1}{2}}(\partial\Omega)$ under the sole assumptions that $\Omega$ is of class $C^{0,1}$. We can also represent in Fourier series the boundary operators involved. We mention \cite{verchota} for an alternative approach to problem \eqref{Neumann_problem} on Lipschitz domains.

Since we have not been able to find problems \eqref{BSM} and \eqref{BSL} in the literature, we believe that it is worth including in the present paper also some information on their spectral behavior, which may have a certain interest on its own. In particular, we prove Lipschitz continuity results for the functions $\mu\mapsto\lambda_j(\mu)$ and $\lambda\mapsto\mu_j(\lambda)$ and we show that problems \eqref{DBS} and \eqref{NBS} can be seen as limiting problems for \eqref{BSM} and \eqref{BSL} as $\mu\rightarrow-\infty$ and $\lambda\rightarrow-\infty$, respectively. We also perform a complete study of the eigenvalues in the unit ball in $\mathbb R^N$ for $\sigma=0$, and we discuss the asymptotic behavior of $\lambda_j(\mu)$ and $\mu_j(\lambda)$ on smooth domains when $j\rightarrow+\infty$. Finally, we briefly discuss problems \eqref{BSM} and \eqref{BSL} also when $\mu> 0$ and $\lambda > \eta_1$. %Also in this case, we will see that the eigenvalues of \eqref{DBS} and \eqref{NBS} play an important role in the description of the spectrum of \eqref{BSL} and \eqref{BSM} respectively when $\lambda\geq \eta_1$ and $\mu\geq 0$, being $\eta_1$ the first eigenvalue of \eqref{DBS}.

Our approach is similar to that developed by G. Auchmuty in \cite{auchmuty_steklov} for the trace space of $H^1(\Omega)$, based on the classical second order Steklov problem
$$
\begin{cases}
\Delta u=0\,, & {\rm in\ }\Omega,\\
\frac{\partial u}{\partial\nu}=\lambda u\,, & {\rm on\ }\partial\Omega.
\end{cases}
$$
We also refer to \cite{torres} for related results.

This paper is organized as follows. In Section \ref{pre} we introduce some notation and discuss a few preliminary results. In Section \ref{problems} we discuss problems \eqref{BSM} and \eqref{BSL} when $\mu<0$ and $\lambda<\eta_1$. In Section \ref{traceproblem} we define the spaces $\mathcal S^{\frac{3}{2}}(\partial\Omega)$ and $\mathcal S^{\frac{1}{2}}(\partial\Omega)$ and the representation theorems for the trace spaces of $H^2(\Omega)$. In Subsection \ref{dir_prob_2} we prove a representation result for the solutions of the biharmonic Dirichlet problem.  In Appendix \ref{ball} we provide a complete description of problems \eqref{BSM} and \eqref{BSL} on the unit ball for $\sigma=0$. In Appendix \ref{asymptotic} we briefly discuss asymptotic laws for the eigenvalues. In Appendix \ref{appendixA} we discuss problems \eqref{BSM} and \eqref{BSL} when $\mu> 0$ and $\lambda>\eta_1$.

 %In Section \ref{problems2} we prove the representation results for the solutions of the Dirichlet problem.

%%%%%%%%%%%%%%%%%%%%%%%%%%%%%%%%%%%%%%%%%%%%%%%%%%%%%%%%%%%%%%%%%%%%%%%%%%%%%%%%%%%%%%%%%%%%%%%%%%%%%%%%%%%%%%%%%%%%%%%%%%%%%%%%%%%%%%%%%%%%%%%%%%%%%%%%%%%
%%%%%%%%%%%%%%%%%%%%%%%%%%%%%%%%%%%%%%%%%%%%%%%%%%%%%%%%%%%%%%%%%%%%%%%%%%%%%%%%%%%%%%%%%%%%%%%%%%%%%%%%%%%%%%%%%%%%%%%%%%%%%%%%%%%%%%%%%%%%%%%%%%%%%%%%%%%

%%%%%%%%%%%%%%%%%%%%%%%%%%%%%%%%%%%%%%%%%%%%%%%%%%%%%%%%%%%%              PRELIMINARIES                   %%%%%%%%%%%%%%%%%%%%%%%%%%%%%%%%%%%%%%%%%%%%%%%%%

%%%%%%%%%%%%%%%%%%%%%%%%%%%%%%%%%%%%%%%%%%%%%%%%%%%%%%%%%%%%%%%%%%%%%%%%%%%%%%%%%%%%%%%%%%%%%%%%%%%%%%%%%%%%%%%%%%%%%%%%%%%%%%%%%%%%%%%%%%%%%%%%%%%%%%%%%%%
%%%%%%%%%%%%%%%%%%%%%%%%%%%%%%%%%%%%%%%%%%%%%%%%%%%%%%%%%%%%%%%%%%%%%%%%%%%%%%%%%%%%%%%%%%%%%%%%%%%%%%%%%%%%%%%%%%%%%%%%%%%%%%%%%%%%%%%%%%%%%%%%%%%%%%%%%%%

\section{Preliminaries and notation}\label{pre}
For a bounded domain (i.e., a bounded open connected set) $\Omega$ in $\mathbb R^N$, we denote by $H^1(\Omega)$ the standard Sobolev space of functions in $L^2(\Omega)$ with all weak derivatives of the first order in $L^2(\Omega)$ endowed with its standard norm $\|u\|_{H^1(\Omega)}:=\left(\|\nabla u\|^2_{L^2(\Omega)}+\|u\|^2_{L^2(\Omega)}\right)^{\frac{1}{2}}$ for all $u\in H^1(\Omega)$. Note that in this paper we consider $L^2(\Omega)$ as a space of real-valued functions and we always assume $N\geq 2$.

 By $H^2(\Omega)$ we denote the standard Sobolev space of functions in $L^2(\Omega)$ with all weak derivatives of the first and second order in $L^2(\Omega)$ endowed with the norm $\|u\|_{H^2(\Omega)}:=\left(\|D^2u\|^2_{L^2(\Omega)}+\|u\|^2_{L^2(\Omega)}\right)^{\frac{1}{2}}$ for all $u\in H^2(\Omega)$. We denote by $H_0^1(\Omega)$ the closure of $C^{\infty}_c(\Omega)$ in $H^1(\Omega)$ and by $H^2_0(\Omega)$ the closure of $C^{\infty}_c(\Omega)$ in $H^2(\Omega)$. The space $C^{\infty}_c(\Omega)$ is the space of all functions in $C^{\infty}(\Omega)$ with compact support in $\Omega$. If the boundary is sufficiently regular (e.g., if $\Omega$ is of class $C^{0,1}$), the norm defined by $\sum_{|\alpha|\leq 2}\|D^{\alpha}u\|_{L^2(\Omega)}$ is a norm on $H^2(\Omega)$ equivalent to the standard one.

By definition, a domain of class $C^{0,1}$ is such that locally around each point of its boundary it can be described as the sub-graph of a Lipschitz continuous function. Also, we shall say that $\Omega$ is of class $C^{2,1}$ if locally around each point of its boundary the domain can be described as the sub-graph of a function of class $C^{2,1}$.

By $(\cdot,\cdot)_{\partial\Omega}$ we denote the standard scalar product of $L^2(\partial\Omega)$, namely
$$
(u,v)_{\partial\Omega}:=\int_{\partial\Omega}uv d\sigma\,,\ \ \ \forall u,v\in L^2(\partial\Omega).
$$
We denote by $\gamma_0(u)\in L^2(\partial\Omega)$ the trace of $u$ and by $\gamma_1(u)\in L^2(\partial\Omega)$ the normal derivative of $u$, that is, $\gamma_1(u)=\nabla u\cdot\nu$. By $\Gamma$ we denote the total trace operator from $H^2(\Omega)$ to $L^2(\partial\Omega)\times L^2(\partial\Omega)$ defined by
$$
\Gamma_2(u)=(\gamma_0(u),\gamma_1(u))\,,
$$
for all $u\in H^2(\Omega)$. The operator $\Gamma$ is compact. If $\Omega$ is of class $C^{2,1}$ then $\Gamma$ is a linear and continuous operator from $H^2(\Omega)$ onto $H^{\frac{3}{2}}(\partial\Omega)\times H^{\frac{1}{2}}(\partial\Omega)$ admitting a right continuous inverse. We refer to e.g., \cite{necas} for more details.

Here $H^{\frac{3}{2}}(\partial\Omega), H^{\frac{1}{2}}(\partial\Omega)$ denote the standard Sobolev spaces of fractional order  (see e.g., \cite{grisvard,necas} for more details).
For any $\sigma\in\big(-\frac{1}{N-1},1\big)$, $\mu,\lambda\in\mathbb R$ and $u,v\in H^2(\Omega)$ we set
\begin{equation*}%\label{Qsigma}
\mathcal Q_{\sigma}(u,v)=(1-\sigma)\int_{\Omega}D^2u:D^2v dx+\sigma\int_{\Omega}\Delta u\Delta v dx,
\end{equation*}
\begin{equation*}%\label{quadratic_form_M}
\mathcal{Q}_{\mu,D}(u,v)=\mathcal Q_{\sigma}(u,v)-\mu(\gamma_0(u),\gamma_0(v))_{\partial\Omega},
\end{equation*}
and
\begin{equation*}%\label{quadratic_form_L}
\mathcal{Q}_{\lambda,N}(u,v)=\mathcal Q_{\sigma}(u,v)-\lambda(\gamma_1(u),\gamma_1(v))_{\partial\Omega},
\end{equation*}
where $D^2u:D^2 v=\sum_{i,j=1}^N\frac{\partial^2u}{\partial x_i\partial x_j}\frac{\partial^2 v}{\partial x_i\partial x_j}$ denotes the Frobenius product of the Hessians matrices. Note that if $\sigma\in\big(-\frac{1}{N-1},1\big)$, then the quadratic form $\mathcal Q_{\sigma}$ is coercive in $H^2(\Omega)$ and the norm $\left(\mathcal Q_{\sigma}(u,u)+\|u\|_{L^2(\Omega)}^2\right)^{\frac{1}{2}}$ is equivalent to the standard norm of $H^2(\Omega)$, see e.g., \cite{chasman2}.

It is easy to see that if $\Omega$ is a bounded domain of class $C^{0,1}$ then the space $H^2(\Omega)$ can be endowed with the equivalent norm
\begin{equation*}%\label{equivalent}
\left(\|D^2u\|^2_{L^2(\Omega)}+\|\gamma_0(u)\|^2_{L^2(\partial\Omega)}\right)^{\frac{1}{2}}.
\end{equation*}

We set
\begin{equation*}%\label{HD}
\mathcal{H}^2_{0,D}(\Omega)=\left\{u\in H^2(\Omega):\gamma_0(u)=0\right\}
\end{equation*}
and
\begin{equation*}%\label{HN}
\mathcal{H}^2_{0,N}(\Omega)=\left\{u\in H^2(\Omega):\gamma_1(u)=0\right\}.
\end{equation*}
 The spaces $\mathcal{H}_{0,D}^2(\Omega)$ and  $\mathcal{H}_{0,N}^2(\Omega)$ are closed subspaces of $H^2(\Omega)$ and $\mathcal{H}^2_{0,N}(\Omega)\cap \mathcal{H}^2_{0,D}(\Omega)=H^2_0(\Omega)$. We also note that $\mathcal{H}^2_{0,D}=H^2(\Omega)\cap H^1_0(\Omega)$.

It is useful to recall the so-called biharmonic Green formula 
\begin{multline}\label{biharmonc_green}
\int_{\Omega}D^2\psi:D^2\varphi dx=\int_{\Omega}(\Delta^2\psi)\varphi dx+\int_{\partial\Omega}\frac{\partial^2\psi}{\partial\nu^2}\frac{\partial\varphi}{\partial\nu}d\sigma\\
-\int_{\partial\Omega}\left({\rm div}_{\partial\Omega}(D^2\psi\cdot\nu)_{\partial\Omega}+\frac{\partial\Delta \psi}{\partial\nu}\right)\varphi d\sigma,
\end{multline}
valid for all sufficiently smooth $\psi,\varphi$, see \cite{arrietalamberti1}.

The biharmonic functions in $H^2(\Omega)$ are defined as those functions $u\in H^2(\Omega)$ such that $\int_{\Omega}D^2u:D^2\varphi dx=0$ for all $\varphi\in H^2_0(\Omega)$, or equivalently, thanks to \eqref{biharmonc_green}, those functions $u\in H^2(\Omega)$ such that $\int_{\Omega}\Delta u\Delta\varphi dx=0$ for all $\varphi\in H^2_0(\Omega)$ . We denote by $\mathcal B_N(\Omega)$ the space of biharmonic functions with zero normal derivative, that is the orthogonal complement of $\mathcal H^2_{0,N}(\Omega)$ in $H^2_0(\Omega)$ with respect to $\mathcal Q_{\sigma}$:
\begin{equation}\label{BN}
\mathcal B_N(\Omega):=\left\{u\in \mathcal H^2_{0,N}(\Omega):\mathcal Q_{\sigma}(u,\varphi)=0\,,\forall\varphi\in H^2_0(\Omega)\right\}.
\end{equation}
By formula \eqref{biharmonc_green} and a standard approximation we deduce that
\begin{equation}\label{BN2}
\mathcal B_N(\Omega):=\left\{u\in \mathcal H^2_{0,N}(\Omega):\int_{\Omega}\Delta u\Delta\varphi dx=0\,,\forall\varphi\in H^2_0(\Omega)\right\}.
\end{equation}
We note that $\mathcal B_N(\Omega)$ is the space of the biharmonic functions in $\mathcal H^2_{0,N}(\Omega)$. Thus
$$
\mathcal H^2_{0,N}(\Omega)= H^2_0(\Omega)\oplus\mathcal B_N(\Omega).
$$
Analogously, we denote by $\mathcal B_D$ the space of biharmonic functions with zero boundary trace, that is the orthogonal complement of $\mathcal H^2_{0,D}(\Omega)$ in $H^2_0(\Omega)$ with respect to $\mathcal Q_{\sigma}$:
\begin{equation}\label{BD}
\mathcal B_D(\Omega):=\left\{u\in \mathcal H^2_{0,D}(\Omega):\mathcal Q_{\sigma}(u,\varphi)=0\,,\forall\varphi\in H^2_0(\Omega)\right\}.
\end{equation}
By formula \eqref{biharmonc_green} and standard approximation we deduce that
\begin{equation}\label{BD2}
\mathcal B_D(\Omega):=\left\{u\in \mathcal H^2_{0,D}(\Omega):\int_{\Omega}\Delta u\Delta\varphi dx=0\,,\forall\varphi\in H^2_0(\Omega)\right\}.
\end{equation}
We note that $\mathcal B_D(\Omega)$ is the space of biharmonic functions in $\mathcal H^2_{0,D}(\Omega)$. Thus
$$
\mathcal H^2_{0,D}(\Omega)= H^2_0(\Omega)\oplus\mathcal B_D(\Omega).
$$

Finally, by $\mathbb N$ we denote the set of positive natural numbers and by $\mathbb N_0$ the set $\mathbb N\cup\left\{0\right\}$.

%%%%%%%%%%%%%%%%%%%%%%%%%%%%%%%%%%%%%%%%%%%%%%%%%%%%%%%%%%%%%%%%%%%%%%%%%%%%%%%%%%%%%%%%%%%%%%%%%%%%%%%%%%%%%%%%%%%%%%%%%%%%%%%%%%%%%%%%%%%%%%%%%%%%%%%%%%%
%%%%%%%%%%%%%%%%%%%%%%%%%%%%%%%%%%%%%%%%%%%%%%%%%%%%%%%%%%%%%%%%%%%%%%%%%%%%%%%%%%%%%%%%%%%%%%%%%%%%%%%%%%%%%%%%%%%%%%%%%%%%%%%%%%%%%%%%%%%%%%%%%%%%%%%%%%%

%%%%%%%%%%%%%%%%%%%%%%%%%%%%%%%%%%%%%%%%%%%%%%%%%%%%%              MULTIPARAMETER STEKLOV PBS             %%%%%%%%%%%%%%%%%%%%%%%%%%%%%%%%%%%%%%%%%%%%%%%%%

%%%%%%%%%%%%%%%%%%%%%%%%%%%%%%%%%%%%%%%%%%%%%%%%%%%%%%%%%%%%%%%%%%%%%%%%%%%%%%%%%%%%%%%%%%%%%%%%%%%%%%%%%%%%%%%%%%%%%%%%%%%%%%%%%%%%%%%%%%%%%%%%%%%%%%%%%%%
%%%%%%%%%%%%%%%%%%%%%%%%%%%%%%%%%%%%%%%%%%%%%%%%%%%%%%%%%%%%%%%%%%%%%%%%%%%%%%%%%%%%%%%%%%%%%%%%%%%%%%%%%%%%%%%%%%%%%%%%%%%%%%%%%%%%%%%%%%%%%%%%%%%%%%%%%%%

\section{Multi-parameter Steklov problems}\label{problems}
In this section we provide the appropriate weak formulations of problems \eqref{BSM} and \eqref{BSL}. In particular we prove that both problems have discrete spectrum provided $\mu<0$ and $\lambda<\eta_1$, respectively. Here $\eta_1$ is the first eigenvalue of problem \eqref{weak_DBS} below, which is the weak formulation of \eqref{DBS}. We remark that $\eta_1>0$ and that $\xi_1=0$ is the first eigenvalue of problem \eqref{NBS}, hence the condition $\mu<0$ reads $\mu<\xi_1$. We also provide  a variational characterization of the eigenvalues.

Through all this section $\Omega$ will be a bounded domain of class $C^{0,1}$ and $\sigma\in\big(-\frac{1}{N-1},1\big)$ will be fixed.
\subsection{The \eqref{DBS} and \eqref{NBS} problems}

Before analyzing problems \eqref{BSM} and \eqref{BSL} we need to recall a few facts about problems \eqref{DBS} and \eqref{NBS}.

Problem \eqref{DBS} is understood in the weak sense as follows:
\begin{equation}\label{weak_DBS}
\int_{\Omega}(1-\sigma)D^2 v:D^2\varphi +\sigma\Delta v\Delta\varphi dx=\eta\int_{\partial\Omega}\frac{\partial v}{\partial\nu}\frac{\partial\varphi}{\partial\nu} d\sigma\,,\ \ \ \forall\varphi\in \mathcal{H}^2_{0,D}(\Omega),
\end{equation}
in the unknowns $v\in \mathcal{H}^2_{0,D}(\Omega)$, $\eta\in\mathbb R$. Note that formulation \eqref{weak_DBS} is justified by formula \eqref{biharmonc_green}. Indeed, by applying formula \eqref{biharmonc_green}, one can easily see that if $v$ is a smooth solution to problem \eqref{weak_DBS}, then $v$ is a solution to the classical problem \eqref{DBS} (the same considerations can be done for all other problems discussed in this paper).

We have the following theorem.
\begin{thm}\label{thm_DBS}
Let $\Omega$ be a bounded domain in $\mathbb R^N$ of class $C^{0,1}$ and let $\sigma\in\big(-\frac{1}{N-1},1\big)$. The eigenvalues of problem \eqref{weak_DBS} have finite multiplicity and are given by a non-decreasing sequence of positive real numbers $\eta_j$ defined by
\begin{equation}\label{minmax_DBS}
\eta_j=\min_{\substack{V\subset \mathcal{H}^2_{0,D}(\Omega)\\{\rm dim}V=j}}\max_{\substack{v\in V\\u\ne 0}}\frac{\mathcal Q_{\sigma}(v,v)}{\int_{\partial\Omega}\left(\frac{\partial v}{\partial\nu}\right)^2d\sigma},
\end{equation}
where each eigenvalue is repeated according to its multiplicity. Moreover, there exists a Hilbert basis $\left\{v_j\right\}_{j=1}^{\infty}$ of $\mathcal B_D(\Omega)$ of eigenfunctions $v_j$ associated with the eigenvalues $\eta_j$. Finally, by normalizing the eigenfunctions $v_j$ with respect to $\mathcal Q_{\sigma}$ for all $j\geq 1$, the functions $\hat v_j:=\sqrt{\eta_j}\gamma_1(v_j)$ define a Hilbert basis of $L^2(\partial\Omega)$ with respect to its standard scalar product.
\end{thm}

Problem \eqref{NBS} is understood in the weak sense as follows:
\begin{equation}\label{weak_NBS}
\int_{\Omega}(1-\sigma)D^2 u:D^2\varphi +\sigma\Delta u\Delta\varphi dx=\xi\int_{\partial\Omega}u\varphi d\sigma\,,\ \ \ \forall\varphi\in \mathcal{H}^2_{0,N}(\Omega),
\end{equation}
in the unknowns $u\in \mathcal{H}^2_{0,N}(\Omega)$, $\xi\in\mathbb R$. We have the following theorem.
\begin{thm}\label{thm_NBS}
Let $\Omega$ be a bounded domain in $\mathbb R^N$ of class $C^{0,1}$ and let $\sigma\in\big(-\frac{1}{N-1},1\big)$. The eigenvalues of problem \eqref{weak_NBS} have finite multiplicity and are given by a non-decreasing sequence of non-negative real numbers $\xi_j$ defined by
\begin{equation*}%\label{minmax_NBS}
\xi_j=\min_{\substack{U\subset \mathcal{H}^2_{0,N}(\Omega)\\{\rm dim}U=j}}\max_{\substack{u\in U\\u\ne 0}}\frac{\mathcal Q_{\sigma}(u,u)}{\int_{\partial\Omega}u^2d\sigma},
\end{equation*}
where each eigenvalue is repeated according to its multiplicity. The first eigenvalue $\xi_1=0$ has multiplicity one and the corresponding eigenfunctions are the constant functions on $\Omega$. Moreover, there exists a Hilbert basis $\left\{u_j\right\}_{j=1}^{\infty}$ of $\mathcal B_N(\Omega)$ of eigenfunctions $u_j$ associated with the eigenvalues $\xi_j$. Finally, by normalizing the eigenfunctions $u_j$ with respect to $\mathcal Q_{\sigma}$ for all $j\geq 2$, the functions $\hat u_j:=\sqrt{\xi_j}\gamma_0(u_j)$, $j\geq 2$, and $\hat u_1=|\partial\Omega|^{-1/2}$ define a Hilbert basis of $L^2(\partial\Omega)$ with respect to its standard scalar product.
\end{thm}

The proofs of Theorems \ref{thm_DBS} and \ref{thm_NBS} can be carried out exactly as those of Theorems \ref{main-BSM} and \ref{main-BSL} presented in Subsections \ref{sub_BSM} and \ref{sub_BSL}. Note that the condition $\sigma\in\big(-\frac{1}{N-1},1\big)$ is used to guarantee the coercivity of the form $\mathcal Q_{\sigma}$ discussed in the previous section.

%{\color{red}In this subsection shall we just recall the results as done above, or add some proof or more details?  Shall we use other letters than $v,u$ for DBS and NBS problems?}

%%%%%%%%%%%%%%%%%%%%%%%%%%%%%%%%%%%%%%%%%%%%%%%%%%%%%%              BIHARMONIC STEKLOV MU                   %%%%%%%%%%%%%%%%%%%%%%%%%%%%%%%%%%%%%%%
%%%%%%%%%%%%%%%%%%%%%%%%%																																																		%%%%%%%%%%%%%%%%%%%%%%%%%%%

\subsection{The $BS_{\mu}$ eigenvalue problem}\label{sub_BSM}
For any $\mu\in\mathbb R$, the weak formulation of problem \eqref{BSM} reads
\begin{equation}\label{weak_BSM}
\int_{\Omega}(1-\sigma)D^2 v:D^2\varphi +\sigma\Delta v\Delta\varphi dx-\mu\int_{\partial\Omega}v\varphi d\sigma=\lambda(\mu)\int_{\partial\Omega}\frac{\partial v}{\partial\nu}\frac{\partial\varphi}{\partial\nu} d\sigma\,,\ \ \ \forall\varphi\in H^2(\Omega),
\end{equation}
in the unknowns $v\in H^2(\Omega)$, $\lambda(\mu)\in\mathbb R$, and can be re-written as
\begin{equation*}
\mathcal{Q}_{\mu,D}(v,\varphi)=\lambda(\mu)\left(\gamma_1(v),\gamma_1(\varphi)\right)_{\partial\Omega}\,,\ \ \ \forall\varphi\in H^2(\Omega).
\end{equation*}
We prove that for all $\mu<0$, problem \eqref{BSM} admits an increasing sequence of eigenvalues of finite multiplicity diverging to $+\infty$ and the corresponding eigenfunctions form a basis of $H^2_{\mu,D}(\Omega)$, where $H^2_{\mu,D}(\Omega)$ denotes the orthogonal complement of $\mathcal{H}^2_{0,N}(\Omega)$ in $H^2(\Omega)$ with respect to the scalar product $\mathcal{Q}_{\mu,D}$, namely
\begin{equation}\label{HM}
H^2_{\mu,D}(\Omega)=\left\{v\in H^2(\Omega):\mathcal{Q}_{\mu,D}(v,\varphi)=0\,,\ \forall\varphi\in \mathcal{H}^2_{0,N}(\Omega)\right\}.
\end{equation}
%We postpone the discussion on the case $\mu\geq 0$ to the Appendix \ref{appendixA}.
To do so, we recast problem \eqref{weak_BSM} in the form of an eigenvalue problem for a compact self-adjoint operator acting on a Hilbert space. We consider on $H^2(\Omega)$ the equivalent norm %(since $\mu<0$)
$$
\|v\|_{\mu,D}^2=\mathcal{Q}_{\mu,D}(v,v)
$$
which is associated with the scalar product defined by 
$$
\langle v,\varphi\rangle_{\mu,D}=\mathcal{Q}_{\mu,D}(v,\varphi),
$$
for all $v,\varphi\in H^2(\Omega)$. Then we define the operator $B_{\mu,D}$ from $H^2(\Omega)$ to its dual $(H^2(\Omega))'$ by setting
$$
B_{\mu,D}(v)[\varphi]=\langle v,\varphi\rangle_{\mu,D}\,,\ \ \ \forall v,\varphi\in H^2(\Omega).
$$
By the Riesz Theorem it follows that $B_{\mu,D}$ is a surjective isometry. Then we consider the operator $J_1$ from $H^2(\Omega)$ to $(H^2(\Omega))'$ defined by
\begin{equation}\label{J1}
J_1(v)[\varphi]=(\gamma_1(v),\gamma_1(\varphi))_{\partial\Omega}\,,\ \ \ \forall v,\varphi\in H^2(\Omega).
\end{equation}
The operator $J_1$ is compact since $\gamma_1$ is a compact operator from $H^2(\Omega)$ to $L^2(\partial\Omega)$. Finally, we set 
\begin{equation}\label{Tm}
T_{\mu,D}=B_{\mu,D}^{(-1)}\circ J_1.
\end{equation}
From the compactness of $J_1$ and the boundedness of $B_{\mu,D}$ it follows that $T_{\mu,D}$ is a compact operator from $H^2(\Omega)$ to itself. Moreover, $\langle T_{\mu,D}(v),\varphi\rangle_{\mu,D}=(\gamma_1(v),\gamma_1(\varphi))_{\partial\Omega}$, for all $u,\varphi\in H^2(\Omega)$, hence $T_{\mu,D}$ is self-adjoint. Note that  ${\rm Ker}\,T_{\mu,D}={\rm Ker}\,J_1=\mathcal{H}^2_{0,N}(\Omega)$ and the non-zero eigenvalues of $T_{\mu,D}$ coincide with the reciprocals of the eigenvalues of \eqref{weak_BSM}, the eigenfunctions being the same.

We are now ready to prove the following theorem.

\begin{thm}\label{main-BSM}
Let $\Omega$ be a bounded domain in $\mathbb R^N$ of class $C^{0,1}$ and let $\sigma\in\big(-\frac{1}{N-1},1\big)$. Let $\mu<0$. Then the eigenvalues of \eqref{weak_BSM} have finite multiplicity and are given by a non-decreasing sequence of positive real numbers $\left\{\lambda_j(\mu)\right\}_{j=1}^{\infty}$ defined by
\begin{equation}\label{minmaxM}
\lambda_j(\mu)=\min_{\substack{V\subset H^2(\Omega)\\{\rm dim}V=j}}\max_{\substack{v\in V\\\frac{\partial v}{\partial\nu}\ne 0}}\frac{\mathcal{Q}_{\mu,D}(v,v)}{\int_{\partial\Omega}\left(\frac{\partial v}{\partial\nu}\right)^2d\sigma},
\end{equation}
where each eigenvalue is repeated according to its multiplicity. 

Moreover there exists a basis $\left\{v_{j,\mu}\right\}_{j=1}^{\infty}$ of $H^2_{\mu,D}(\Omega)$ of eigenfunctions $v_{j,\mu}$ associated with the eigenvalues $\lambda_j(\mu)$.

By normalizing the eigenfunctions $v_{j,\mu}$ with respect to $\mathcal{Q}_{\mu,D}$, the functions defined by $\left\{\hat v_{j,\mu}\right\}_{j=1}^{\infty}:=\left\{\sqrt{\lambda_{j}(\mu)}\gamma_1(v_{j,\mu})\right\}_{j=1}^{\infty}$ form a Hilbert basis of $L^2(\partial\Omega)$ with respect to its standard scalar product.
\end{thm}

\begin{proof}
  Since ${\rm Ker}\, T_{\mu,D}=\mathcal H^2_{0,N}$, by the Hilbert-Schmidt Theorem applied to the compact self-adjoint operator $T_{\mu,D}$ it follows that $T_{\mu,D}$ admits an increasing sequence of positive eigenvalues $\left\{q_j\right\}_{j=1}^{\infty}$ bounded from above, converging to zero and a corresponding Hilbert basis $\left\{v_{j,\mu}\right\}_{j=1}^{\infty}$ of eigenfunctions of $H^2_{\mu,D}(\Omega)$. Since $q\ne 0$ is an eigenvalue of $T_{\mu,D}$ if and only if $\lambda=\frac{1}{q}$ is an eigenvalue of \eqref{weak_BSM} with the same eigenfunctions, we deduce the validity of the first part of the statement. In particular, formula \eqref{minmaxM} follows from the standard min-max formula for the eigenvalues of compact self-adjoint operators. Note that $\lambda_1(\mu)>0$, since $Q_{\mu,D}(v,v)=0$ if and only if $v=0$.

To prove the final part of the theorem, we recast problem \eqref{weak_BSM} into an eigenvalue problem for the compact self-adjoint operator $T_{\mu,D}'=\gamma_1\circ B_{\mu,D}^{(-1)}\circ J_1'$, where $J_1'$ denotes the map from $L^2(\partial\Omega)$ to the dual of $H^2(\Omega)$ defined by
$$
J_1'(v)[\varphi]=(v,\gamma_1(\varphi))_{\partial\Omega}\,,\ \ \ \forall v\in L^2(\partial\Omega),\varphi\in H^2(\Omega).
$$
We apply again the Hilbert-Schmidt Theorem and observe that $T_{\mu,D}$ and $T_{\mu,D}'$ admit the same non-zero eigenvalues and that the eigenfunctions of $T_{\mu,D}'$ are exactly the normal derivatives of the eigenfunctions of $T_{\mu,D}$. From \eqref{weak_BSM} we deduce that if the eigenfunctions $v_{j,\mu}$ of $T_{\mu,D}$ are normalized by $\mathcal{Q}_{\mu,D}(v_{j,\mu},v_{k,\mu})=\delta_{jk}$, where $\delta_{jk}$ is the Kronecker symbol, then the normalization of the traces of their normal derivatives in $L^2(\partial\Omega)$ are obtained by multiplying $\gamma_1(v_{j,\mu})$ by $\sqrt{\lambda_j(\mu)}$. This concludes the proof.
\end{proof}

We present now a few results on the behavior of the eigenvalues of \eqref{weak_BSM}  for $\mu\in(-\infty,0)$, in particular we prove a Lipschitz continuity result for the eigenvalues $\lambda_j(\mu)$ with respect to $\mu$ and find their limits as $\mu\rightarrow-\infty$. 

\begin{thm}\label{Lipschitz-continuity-M}
For any $j\in\mathbb N$ and $\delta>0$, the function $\lambda_j:(-\infty,-\delta]\rightarrow (0,+\infty)$ which takes $\mu\in(-\infty,-\delta]$ to $\lambda_j(\mu)\in(0,+\infty)$ is Lipschitz continuous on $(-\infty,-\delta]$. 
\end{thm}
\begin{proof}
Without loss of generality we assume that $\mu_1,\mu_2\in(-\infty,-\delta]$  and  that $\mu_1<\mu_2$. Let $v\in H^2(\Omega)$. Then
\begin{multline*}
0\leq\frac{\mathcal Q_{\mu_1,D}(v,v)}{\int_{\partial\Omega}\left(\frac{\partial v}{\partial\nu}\right)^2d\sigma}-\frac{\mathcal Q_{\mu_2,D}(v,v)}{\int_{\partial\Omega}\left(\frac{\partial v}{\partial\nu}\right)^2d\sigma}=(\mu_2-\mu_1)\frac{\int_{\partial\Omega}v^2d\sigma}{\int_{\partial\Omega}\left(\frac{\partial v}{\partial\nu}\right)^2d\sigma}\\
\leq -\frac{(\mu_2-\mu_1)}{\mu_1}\frac{\mathcal Q_{\mu_1,D}(v,v)}{\int_{\partial\Omega}\left(\frac{\partial v}{\partial\nu}\right)^2d\sigma}.
\end{multline*}
Hence
\begin{equation}\label{lip-M-1}
\frac{\mathcal Q_{\mu_1,D}(v,v)}{\int_{\partial\Omega}\left(\frac{\partial v}{\partial\nu}\right)^2d\sigma}\geq\frac{\mathcal Q_{\mu_2,D}(v,v)}{\int_{\partial\Omega}\left(\frac{\partial v}{\partial\nu}\right)^2d\sigma}
\end{equation}
and
\begin{equation}\label{lip-M-2}
\frac{\mathcal Q_{\mu_2,D}(v,v)}{\int_{\partial\Omega}\left(\frac{\partial v}{\partial\nu}\right)^2d\sigma}\geq\frac{\mathcal Q_{\mu_1,D}(v,v)}{\int_{\partial\Omega}\left(\frac{\partial v}{\partial\nu}\right)^2d\sigma}\left(1+\frac{(\mu_2-\mu_1)}{\mu_1}\right)
\end{equation}

By taking the infimum and the supremum over $j$ dimensional subspaces of $H^2(\Omega)$ into \eqref{lip-M-1} and \eqref{lip-M-2}, and by \eqref{minmaxM}, we get
\begin{equation*}
|\lambda_j(\mu_1)-\lambda_j(\mu_2)|\leq\frac{\lambda_j(\mu_1)}{|\mu_1|}|\mu_2-\mu_1|\leq\lambda_j(\mu_1)\frac{|\mu_2-\mu_1|}{\delta}.
\end{equation*}
This concludes the proof.
\end{proof}

We now investigate the behavior of the eigenvalues $\lambda_j(\mu)$ as $\mu\rightarrow-\infty$. First, we need to recall a few facts about the convergence of operators defined on variable spaces. As customary, we consider families of spaces and operators depending on a small parameter $\varepsilon\geq 0$ with $\varepsilon\to 0$. This will be applied later with $\varepsilon=-\frac{1}{\mu}$ and $\mu\to-\infty$.

Let us denote by $\mathcal H_{\varepsilon}$ a family of Hilbert spaces for all $\varepsilon\in[0,\varepsilon_0)$ and assume that there exists a corresponding family of linear operators $E_{\varepsilon}:\mathcal H_0\rightarrow\mathcal H_{\varepsilon}$ such that, for any $u\in\mathcal H_0$ 
$$
\|E_{\varepsilon}(u)\|_{\mathcal H_{\varepsilon}}\rightarrow\|u\|_{\mathcal H_0}\,,\ \ \ {\rm as\ }\varepsilon\rightarrow 0.
$$
We recall the definition of compact convergence of operators in the sense of \cite{vainikko}.
\begin{defn}\label{Ecompact}
We say that a family $\left\{K_{\varepsilon}\right\}_{\varepsilon\in[0,\varepsilon_0)}$ of compact operators $K_{\varepsilon}\in\mathcal L(\mathcal H_{\varepsilon})$  converges compactly to $K_0$ if
\begin{enumerate}[i)]
\item for any $\left\{u_{\varepsilon}\right\}_{\varepsilon\in(0,\varepsilon_0)}$ with $\|u_{\varepsilon}-E_{\varepsilon}(u)\|_{\mathcal H_{\varepsilon}}\rightarrow 0$ as $\varepsilon\rightarrow 0$, then $\|K_{\varepsilon}(u_{\varepsilon})-E_{\varepsilon}(K_0(u))\|_{\mathcal H_{\varepsilon}}\rightarrow 0$ as $\varepsilon\rightarrow 0$;
\item for any $\left\{u_{\varepsilon}\right\}_{\varepsilon\in(0,\varepsilon_0)}$ with $u_{\varepsilon}\in\mathcal H_{\varepsilon}$, $\|u_{\varepsilon}\|_{\mathcal H_{\varepsilon}}=1$, then $\left\{K_{\varepsilon}(u_{\varepsilon})\right\}_{\varepsilon\in(0,\varepsilon_0)}$ is precompact in the sense that for all sequences $\varepsilon_n\rightarrow 0$ there exist a sub-sequence $\varepsilon_{n_k}\rightarrow 0$ and $w\in\mathcal H_0$ such that $\|K_{\varepsilon_{n_k}}(u_{\varepsilon_{n_k}})-E_{\varepsilon_{n_k}}(w)\|_{\mathcal H_{\varepsilon_{n_k}}}\rightarrow 0$ as $k\rightarrow+\infty$.
\end{enumerate}
\end{defn}
We also recall the following theorem, where by spectral convergence of a family of operators we mean the convergence of the eigenvalues and the convergence of the eigenfunctions in the sense of \cite{vainikko}, see also \cite[\S2]{ferrero_lamberti}.
\begin{thm}\label{spectral_convergence}
Let $\left\{K_{\varepsilon}\right\}_{\varepsilon\in[0,\varepsilon_0)}$ be non-negative, compact self-adjoint operators in the Hil\-bert spaces $\mathcal H_{\varepsilon}$. Assume that their eigenvalues are given by $\left\{\sigma_j(\varepsilon)\right\}_{j=1}^{\infty}$. If $K_{\varepsilon}$ compactly converge to $K_0$, then there is spectral convergence of $K_{\varepsilon}$ to $K_0$ as $\varepsilon\rightarrow 0$. In particular, for every $j\in\mathbb N$ $\sigma_j(\varepsilon)\rightarrow\sigma_j(0)\,,\ \ \ {\rm as\ }\varepsilon\rightarrow 0$.
\end{thm}

Let $T_D:\mathcal{H}^2_{0,D}(\Omega)\rightarrow \mathcal{H}^2_{0,D}(\Omega)$ be defined by $T_D=B_D^{(-1)}\circ J_1$, where $B_D$ is the operator from $\mathcal{H}^2_{0,D}(\Omega)$ to its dual $(\mathcal{H}^2_{0,D}(\Omega))'$ given by
\begin{equation*}
B_D(v)[\varphi]=\mathcal Q_{\sigma}(v,\varphi)\,,\ \ \ \forall v,\varphi\in \mathcal{H}^2_{0,D}(\Omega),
\end{equation*}
and $J_1$ is defined in \eqref{J1}. By the Riesz Theorem it follows that $B_D$ is a surjective isometry. The operator $T_D$ is the resolvent operator associated with problem \eqref{weak_DBS} and plays the same role of $T_{\mu,D}$ defined in \eqref{Tm}. In fact, as in the proof of Theorem \ref{main-BSM} it is possible to show that $T_D$ admits an increasing sequence of non-zero eigenvalues $\left\{q_j\right\}_{j=1}^{\infty}$ bounded from above and converging to $0$. Moreover, a number $q\ne 0$ is an eigenvalue of $T_D$ if and only if $\eta=\frac{1}{q}$ is an eigenvalue of \eqref{weak_DBS}, with the same eigenfunctions. 

We have now a family of compact self-adjoint operators $T_{\mu,D}$ each defined on the Hilbert space $H^2(\Omega)$ endowed with the scalar product $\mathcal Q_{\mu,D}$, and the compact self-adjoint operator $T_D$ defined on $\mathcal{H}^2_{0,D}(\Omega)$ endowed with the scalar product $\mathcal Q_{\sigma}$.
%Actually, if a sequence of compact self-adjoint operators $K_{\varepsilon}$ compactly converges to a compact self-adjoint operator $K_0$, not only all eigenvalues converge, but we have also convergence of eigenfunctions in an appropriate sense (in particular, the projections on eigenspaces converge in operator norm). However, for the purposes of the present article, we are just interested in the convergence of the eigenvalues. For more details on compact convergence and spectral convergence of compact self-adjoint operators (on variable Hilbert spaces) we refer to 
We are ready to state and prove the following theorem.
\begin{thm}\label{thm-limit-eta}
The family of operators $\left\{T_{\mu,D}\right\}_{\mu\in(-\infty,0)}$ compactly converges to $T_D$ as $\mu\rightarrow-\infty$. In particular, 
\begin{equation}\label{limit_eta}
\lim_{\mu\rightarrow-\infty}\lambda_j(\mu)=\eta_j,
\end{equation}
for all $j\in\mathbb N$, where $\eta_j$ are the eigenvalues of \eqref{weak_DBS}.
\end{thm}
\begin{proof}
For each $\mu\in(-\infty,0)$ we define the map $E_{\mu}\equiv E:\mathcal{H}^2_{0,D}(\Omega)\rightarrow H^2(\Omega)$ simply by setting $E(u)=u$, for all $u\in\mathcal H^2_{0,D}(\Omega)$.

In view of Definition \ref{Ecompact}, we have to prove that
\begin{enumerate}[i)]
\item if $\left\{u_{\mu}\right\}_{\mu<0}\subset H^2(\Omega)$ and $u\in \mathcal{H}^2_{0,D}(\Omega)$ are such that $\mathcal Q_{\mu,D}(u_{\mu}-u,u_{\mu}-u)\rightarrow 0$ as $\mu\rightarrow-\infty$, then
$$
\mathcal Q_{\mu,D}(T_{\mu,D}(u_{\mu})-T_D(u),T_{\mu,D}(u_{\mu})-T_D(u))\rightarrow 0\,,{\rm\ \ \ as\ }\mu\rightarrow-\infty;
$$
\item if  $\left\{u_{\mu}\right\}_{\mu<0}\subset H^2(\Omega)$  is such that $\mathcal Q_{\mu,D}(u_{\mu},u_{\mu})=1$ for all $\mu<0$, then for every sequence $\mu_n\rightarrow-\infty$ there exists a sub-sequence $\mu_{n_k}\rightarrow-\infty$ and $v\in \mathcal H^2_{0,D}(\Omega)$ such that 
\begin{equation}\label{22}
\mathcal Q_{{\mu_{n_k}},D}(T_{{\mu_{n_k}},D}(u_{{\mu_{n_k}}})-v,T_{{\mu_{n_k}},D}(u_{{\mu_{n_k}}})-v)\rightarrow 0\,,{\rm\ \ \ as\ }k\rightarrow+\infty.
\end{equation}
\end{enumerate}

We start by proving i). By the assumptions in i), it follows that $u_{\mu}$ is uniformly bounded in $H^2(\Omega)$ for $\mu$ in a neighborhood of $-\infty$. Indeed, by definition 
\begin{equation}\label{compact1}
\mathcal Q_{\mu,D}(T_{\mu,D}(u_{\mu}),\varphi)=\int_{\partial\Omega}\frac{\partial u_{\mu}}{\partial\nu}\frac{\partial\varphi}{\partial\nu}d\sigma\,,\ \ \ \forall\varphi\in H^2(\Omega),
\end{equation}
hence, by choosing $\varphi=T_{\mu,D}(u_{\mu})$, we find that the family $\left\{T_{\mu,D}(u_{\mu})\right\}_{\mu<0}$ is bounded in $H^2(\Omega)$. Thus, possibly passing to a sub-sequence, $T_{\mu,D}(u_{\mu})\rightharpoonup v$ in $H^2(\Omega)$, and $\gamma_0(T_{\mu,D}(u_{\mu}))\rightarrow\gamma_0(v)$ in $L^2(\partial\Omega)$, as $\mu\rightarrow-\infty$, which implies that $\gamma_0(v)=0$ since the term $-\mu\int_{\partial\Omega}T_{\mu,D}(u_{\mu})^2d\sigma$ is bounded in $\mu$. Thus $v\in \mathcal{H}^2_{0,D}(\Omega)$.\\
Choosing $\varphi\in \mathcal H^2_{0,D}(\Omega)$ and passing to the limit in \eqref{compact1} we have that
$$
\mathcal Q_{\sigma}(v,\varphi)=\int_{\partial\Omega}\frac{\partial u}{\partial\nu}\frac{\partial\varphi}{\partial\nu}d\sigma\,,\ \ \ \forall\varphi\in \mathcal{H}^2_{0,D}(\Omega),
$$
hence $v=T_D(u)$. Thus $T_{\mu,D}(u_{\mu})\rightharpoonup T_D(u)$ in $H^2(\Omega)$. Moreover, the convergence is stronger because
\begin{multline*}
\lim_{\mu\rightarrow-\infty}\mathcal Q_{\mu,D}(T_{\mu,D}(u_{\mu})-T_D(u),T_{\mu,D}(u_{\mu})-T_D(u))\\
=\lim_{\mu\rightarrow-\infty}\left(\mathcal Q_{\mu,D}(T_{\mu,D}(u_{\mu}),T_{\mu,D}(u_{\mu}))-2\mathcal Q_{\mu,D}(T_{\mu,D}(u_{\mu}),T_D(u))+\mathcal Q_{\mu,D}(T_D(u),T_D(u))\right)\\
=\mathcal Q_{\sigma}(T_D(u),T_D(u))-2\mathcal Q_{\sigma}(T_D(u),T_D(u))+\mathcal Q_{\sigma}(T_D(u),T_D(u))=0,
\end{multline*}
which proves point i).

Note that the equality $\lim_{\mu\rightarrow-\infty}\mathcal Q_{\mu,D}(T_{\mu,D}(u_{\mu}),T_{\mu,D}(u_{\mu}))=\mathcal Q_{\sigma}(T_D(u),T_D(u))$ is a consequence of
\begin{multline*}
\lim_{\mu\rightarrow-\infty}\mathcal Q_{\mu,D}(T_{\mu,D}(u_{\mu}),T_{\mu,D}(u_{\mu}))=\lim_{\mu\rightarrow-\infty}\int_{\partial\Omega}\frac{\partial u_{\mu}}{\partial\nu}\frac{\partial T_{\mu,D}(u_{\mu})}{\partial\nu}d\sigma\\
=\int_{\partial\Omega}\frac{\partial u}{\partial\nu}\frac{\partial T_D(u)}{\partial\nu}d\sigma=\mathcal Q_{\sigma}(T_D(u),T_D(u)).
\end{multline*}
The proof of point ii) is similar. If $\mathcal Q_{\mu,D}(u_{\mu},u_{\mu})=1$, up to sub-sequences $u_{\mu}\rightharpoonup u\in H^2(\Omega)$, $\gamma_0(u_{\mu})\rightarrow\gamma_0(u)$, and $\gamma_1(u_{\mu})\rightarrow\gamma_1(u)$ as $\mu\rightarrow -\infty$. Moreover, $\|\gamma_0(u_{\mu})\|_{L^2(\partial\Omega)}^2\leq-\frac{1}{\mu}$, hence $\|\gamma_0(u_{\mu})\|_{L^2(\partial\Omega)}^2\rightarrow 0$ as $\mu\rightarrow -\infty$. This implies that $\gamma_0(u)=0$ and that $u\in \mathcal{H}^2_{0,D}(\Omega)$. Then it is possible to repeat the same arguments above to conclude the validity of \eqref{22} with $v=T_D(u)$.

Thus $T_{\mu,D}$ compactly converges to $T_D$ and \eqref{limit_eta} follows by Theorem \ref{spectral_convergence}.

\end{proof}

\begin{rem}
We also note that each eigenvalue $\lambda_j(\mu)$ is non-increasing with respect to $\mu$, for $\mu\in(-\infty,0)$. In fact from the Min-Max Principle \eqref{minmaxM} it immediately follows that for all $j\in\mathbb N$, $\lambda_j(\mu_1)\leq\lambda_j(\mu_2)$ if $\mu_1>\mu_2$. 
\end{rem}

Now we consider the behavior of the first eigenvalue as $\mu\rightarrow 0^-$.

\begin{lem}\label{limit_mu_0}
We have
\begin{equation*}%\label{limit_mu_0_eq}
\lim_{\mu\rightarrow 0^-}\lambda_1(\mu)=0
\end{equation*}
\end{lem}
\begin{proof}
Let $p\in\mathbb R^N$ be fixed. From \eqref{minmaxM} we get 
\begin{equation*}
0<\lambda_1(\mu)=\min_{\substack{v\in H^2(\Omega)\\v\ne 0}}\frac{\mathcal{Q}_{\mu,D}(v,v)}{\int_{\partial\Omega}\left(\frac{\partial v}{\partial\nu}\right)^2d\sigma}\leq \frac{\mathcal{Q}_{\mu,D}(p\cdot x,p\cdot x)}{\int_{\partial\Omega}(p\cdot\nu)^2d\sigma}=-\mu\frac{\int_{\partial\Omega}(p\cdot x)^2d\sigma}{\int_{\partial\Omega}(p\cdot\nu)^2d\sigma},
\end{equation*}
for all $\mu\in(-\infty,0)$. By letting $\mu\rightarrow 0^-$ we obtain the result.
\end{proof}

%We note that problem \eqref{weak_BSM} is not well-posed if $\mu=0$. In fact for all $\lambda\in\mathbb R$ there exists $v\in H^2(\Omega)$ such that \eqref{weak_BSM} holds. In fact the couple $(v\equiv 1,\lambda)$ is a solution of \eqref{weak_BSM} with $\mu=0$ for all $\lambda\in\mathbb R$. However, if we quotient out the constants we obtain a well-posed eigenvalue problem (see Appendix \ref{appendixA}).% Note also that the constant function is the eigenfunction associated with the first eigenvalue $\xi_1=0$ of problem \eqref{NBS}, hence $\xi_1$ is the supremum of 
%$$
%\left\{\tilde\mu\in\mathbb R:  {\rm\eqref{BSM}\  is\  well \ posed\  in\ } H^2(\Omega) {\rm \ for \ all\ } \mu<\tilde\mu\right\}.
%$$

%As we shall see in the sequel, the behavior of problem \eqref{BSM} will be affected by the fact that $\mu\in[\xi_j,\xi_{j+1}[$ where $\xi_j$ are the eigenvalues of problem \eqref{NBS}.

%%%%%%%%%%%%%%%%%%%%%%%%%%%%%%%%%%%%%%%%%%%%%%%%%%%%%%              BIHARMONIC STEKLOV LAMBDA                   %%%%%%%%%%%%%%%%%%%%%%%%%%%%%%%%%%%%%%%
%%%%%%%%%%%%%%%%%%%%%%%%%																																																		%%%%%%%%%%%%%%%%%%%%%%%%%%%

\subsection{The $BS_{\lambda}$ eigenvalue problem}\label{sub_BSL}
The weak formulation of problem \eqref{BSL} reads
\begin{equation}\label{weak_BSL}
\int_{\Omega}(1-\sigma)D^2 u:D^2\varphi +\sigma\Delta u\Delta\varphi dx-\lambda\int_{\partial\Omega}\frac{\partial u}{\partial\nu}\frac{\partial\varphi}{\partial\nu}d\sigma=\mu(\lambda)\int_{\partial\Omega}u\varphi d\sigma\,,\ \ \ \forall\varphi\in H^2(\Omega),
\end{equation}
in the unknowns $u\in H^2(\Omega)$, $\mu(\lambda)\in\mathbb R$, and can be re-written as
\begin{equation*}
\mathcal{Q}_{\lambda,N}(u,\varphi)=\mu(\lambda)\left(\gamma_0(u),\gamma_0(\varphi)\right)_{\partial\Omega}\,,\ \ \ \forall\varphi\in H^2(\Omega).
\end{equation*}
 We prove that for all $\lambda<\eta_1$, where $\eta_1$ is the first eigenvalue of \eqref{DBS}, problem \eqref{BSL} admits an increasing sequence of eigenvalues of finite multiplicity diverging to $+\infty$ and the corresponding eigenfunctions form a basis of $H^2_{\lambda,N}(\Omega)$, where $H^2_{\lambda,N}(\Omega)$ denotes the orthogonal complement of $\mathcal{H}^2_{0,D}(\Omega)$ in $H^2(\Omega)$ with respect to $\mathcal{Q}_{\lambda,N}$, that is
\begin{equation}\label{HL}
H^2_{\lambda,N}(\Omega)=\left\{u\in H^2(\Omega):\mathcal{Q}_{\lambda,N}(u,\varphi)=0\,,\ \forall\varphi\in \mathcal{H}^2_{0,D}(\Omega)\right\}.
\end{equation}
Since in general $\mathcal Q_{\lambda,N}$ is not a scalar product, we find it convenient to consider on $H^2(\Omega)$ the norm
\begin{equation}\label{norm_L}
\|u\|_{\lambda,N}^2=\mathcal{Q}_{\lambda,N}(u,u)+b\|\gamma_0(u)\|_{L^2(\Omega)}^2,
\end{equation}
where $b>0$ is a fixed number which is chosen as follows. If $\lambda<0$, no restrictions are required on $b>0$, since the norm $\|\cdot\|_{\lambda,N}$ is equivalent to the standard norm of $H^2(\Omega)$ for all $b>0$. Assume now that $0\leq\lambda<\eta_1$. From Theorem \ref{thm-limit-eta} and Lemma \ref{limit_mu_0} we have that  $(0,\eta_1)\subseteq\lambda_1((-\infty,0))$,
%In this second case we deduce that if 
%\begin{equation}\label{max_mu}
%\tilde\mu=\max\left\{\mu\in]-\infty,0[:\lambda_1(\mu)=\eta_1\right\},
%\end{equation}
%then $\lambda_1(\mu)=\eta_1$ for all $\mu\in]-\infty,\tilde\mu]$.
hence there exists $\mu\in(-\infty,0)$ and $\varepsilon\in(0,1)$ such that $\lambda_1(\mu)=\frac{\lambda+\varepsilon}{1-\varepsilon}<\eta_1$. Then
\begin{multline}\label{constraint_0}
\mathcal Q_{\lambda,N}(u,u)=\varepsilon\mathcal Q_{-1,N}(u,u)+(1-\varepsilon)\mathcal Q_{\lambda_1(\mu),N}(u,u)\\
\geq \varepsilon\mathcal Q_{-1,N}(u,u)+(1-\varepsilon)\mu\|\gamma_0(u)\|^2_{L^2(\partial\Omega)}.
\end{multline}
Thus, by choosing any $b$ satisfying 
\begin{equation}\label{constraint-b}
%b>-\max\left\{\mu\in ]-\infty,0[:\exists\varepsilon\in]0,1[{\rm\ with\ }\lambda_1(\mu)=\frac{\lambda+\varepsilon}{1-\varepsilon}\right\},
b>-(1-\varepsilon)\mu,
\end{equation}
it follows by \eqref{constraint_0} and \eqref{constraint-b} that $\|\cdot\|_{\lambda,N}$ is a norm equivalent to the standard norm of $H^2(\Omega)$.

The norm $\|\cdot\|_{\lambda,N}$ is associated with the scalar product defined by 
\begin{equation}\label{scalar_L}
\langle u,\varphi\rangle_{\lambda,N}=\mathcal{Q}_{\lambda,N}(u,\varphi)+b(\gamma_0(u),\gamma_0(\varphi))_{\partial\Omega},
\end{equation}
for all $u,\varphi\in H^2(\Omega)$.

We now recast problem \eqref{weak_BSL} in the form of an eigenvalue problem for a compact self-adjoint operator acting on a Hilbert space. To do so, we define the operator $B_{\lambda,N}$ from $H^2(\Omega)$ to its dual $(H^2(\Omega))'$ by setting
$$
B_{\lambda,N}(u)[\varphi]=\langle u,\varphi\rangle_{\lambda,N}\,,\ \ \ \forall u,\varphi\in H^2(\Omega).
$$
By the Riesz Theorem it follows that $B_{\lambda,N}$ is a surjective isometry. Then we consider the operator $J_0$ from $H^2(\Omega)$ to $(H^2(\Omega))'$ defined by
\begin{equation}\label{J0}
J_0(u)[\varphi]=(\gamma_0(u),\gamma_0(\varphi))_{\partial\Omega}\,,\ \ \ \forall u,\varphi\in H^2(\Omega).
\end{equation}
The operator $J_0$ is compact since $\gamma_0$ is a compact operator from $H^2(\Omega)$ to $L^2(\partial\Omega)$. Finally, we set 
\begin{equation}\label{Tl}
T_{\lambda,N}=B_{\lambda,N}^{(-1)}\circ J_0.
\end{equation}
From the compactness of $J_0$ and the boundedness of $B_{\lambda,N}$ it follows that $T_{\lambda,N}$ is a compact operator from $H^2(\Omega)$ to itself. Moreover, $
\langle B_{\lambda,N}(u),\varphi\rangle_{\lambda,N}=(\gamma_0(u),\gamma_0(\varphi))_{\partial\Omega}$, for all $u,\varphi\in H^2(\Omega)$, hence $T_{\lambda,N}$ is self-adjoint. 

Note that ${\rm Ker}\,T_{\lambda,N}={\rm Ker}\,J_0=\mathcal{H}^2_{0,D}(\Omega)$ and the non-zero eigenvalues of $T_{\lambda,N}$ coincide with the reciprocals of the eigenvalues of \eqref{weak_BSL} shifted by $b$, the eigenfunctions being the same.

We are now ready to prove the following theorem.

%We postpone the discussion on the case $\lambda\geq\eta_1$ to the Appendix \ref{appendixA}.

\begin{thm}\label{main-BSL}
Let $\Omega$ be a bounded domain in $\mathbb R^N$ of class $C^{0,1}$, $\sigma\in\big(-\frac{1}{N-1},1\big)$, and $\lambda<\eta_1$. Then the eigenvalues of \eqref{weak_BSL} have finite multiplicity and are given by a non-decreasing sequence of real numbers $\left\{\mu_j(\lambda)\right\}_{j=1}^{\infty}$ defined by
\begin{equation}\label{minmaxL}
\mu_j(\lambda)=\min_{\substack{U\subset H^2(\Omega)\\{\rm dim}U=j}}\max_{\substack{u\in U\\\gamma_0(u)\ne 0}}\frac{\mathcal{Q}_{\lambda,N}(u,u)}{\int_{\partial\Omega}u^2d\sigma},
\end{equation}
where each eigenvalue is repeated according to its multiplicity. Moreover, there exists a Hilbert basis $\left\{u_{j,\lambda}\right\}_{j=1}^{\infty}$ of $H^2_{\lambda,N}(\Omega)$ (endowed with the scalar product \eqref{scalar_L}) of eigenfunctions $u_{j,\lambda}$ associated with the eigenvalues $\mu_j(\lambda)$ and the following statements hold:
%For all $\lambda<\eta_1$, $\mu(\lambda)=0$ is an eigenvalue, and the constant functions belong to the corresponding eigenspace.
\begin{enumerate}[i)]
%\item If $\lambda=0$ 
%If $0<\lambda<\eta_1$ then $\mu_1(\lambda)=M_{\lambda}:=\max\left\{\mu\in]-\infty,0[:\lambda_1(\mu)=\lambda\right\}$. Such maximum always exists. 
\item If $\lambda<0$ then $\mu_1(\lambda)=0$ is an eigenvalue of multiplicity one and the corresponding eigenfunctions are the constant functions. Moreover, if $\tilde u_{j,\lambda}$ denote the normalizations of $u_{j,\lambda}$ with respect to $\mathcal{Q}_{\lambda,N}$ for all $j\geq 2$, the functions $\hat u_{j,\lambda}:=\sqrt{\mu_{j}(\lambda)}\gamma_0(\tilde u_{j,\lambda})$, $j\geq 2$, and $\hat u_{1,\lambda}:=|\partial\Omega|^{-1/2}$ define a Hilbert basis of $L^2(\partial\Omega)$ with respect to its standard scalar product.
\item If $0\leq\lambda<\eta_1$, then $\mu(\lambda)=0$ is an eigenvalue. Moreover, if $\mu_{j_0}(\lambda)$ is the first positive eigenvalue, and $\tilde u_{j,\lambda}$ denote the normalizations of $u_{j,\lambda}$ with respect to $\mathcal{Q}_{\lambda,N}$ for all $j\geq j_0$, and $\left\{\hat u_{j,\lambda}\right\}_{j=1}^{j_0-1}$ denotes a orthonormal basis with respect to the $L^2(\partial\Omega)$ scalar product of the eigenspace associated to $\mu_1(\lambda),...,\mu_{j_0-1}(\lambda)$  restricted to $\partial\Omega$, then the functions $\hat u_{j,\lambda}:=\sqrt{\mu_{j}(\lambda)}\gamma_0(\tilde u_{j,\lambda})$, $j\geq j_0$, and $\left\{\hat u_{j,0}\right\}_{j=1}^{j_0-1}$, define a Hilbert basis of $L^2(\partial\Omega)$ with respect to its standard scalar product. Finally, if $\lambda=0$, then $j_0=N+2$ and  the eigenspace corresponding to $\mu_1(0)=\cdots=\mu_{N+1}(0)=0$ is generated by $\left\{1,x_1,...,x_N\right\}$; if $\lambda>0$, then $\mu_1(\lambda)<0$.
%$\mu_1(0)=\cdots=\mu_{N+1}(0)=0$ is an eigenvalue of multiplicity $N+1$ and the corresponding eigenspace is generated by $\left\{1,x_1,...,x_N\right\}$. 

%\item If $0<\lambda<\eta_1$ and  $j_0\geq 1$ is such that $\mu_j(\lambda)>0$ for $j\geq j_{0}+1$. If $\tilde u_{j,\lambda}$ denote the normalizations of the eigenfunctions $u_{j,\lambda}$ with respect to $\mathcal{Q}_{\lambda,N}$ for all $j\geq j_0+1$, and by taking $\left\{\hat u_{j,\lambda}\right\}_{j=1}^{j_0}$ a orthonormal basis with respect to the $L^2(\partial\Omega)$ scalar product of the space spanned by the restriction to $\partial\Omega$ of $\left\{u_{j,\lambda}\right\}_{j=1}^{j_0}$, the functions $\hat u_{j,\lambda}:=\sqrt{\mu_{j}(\lambda)}\gamma_0(\tilde u_{j,\lambda})$, $j\geq j_0+1$, and $\left\{\hat u_{j,\lambda}\right\}_{j=1}^{j_0}$ define a Hilbert basis of $L^2(\partial\Omega)$ with respect to its standard scalar product.
\end{enumerate}
\end{thm}
\begin{proof}

Since ${\rm Ker}\,J_0=\mathcal{H}^2_{0,D}(\Omega)$, by the Hilbert-Schmidt Theorem applied to $T_{\lambda,N}$ it follows that $T_{\lambda,N}$ admits a non-increasing sequence of positive eigenvalues $\left\{p_j\right\}_{j=1}^{\infty}$ bounded from above, converging to zero and a corresponding Hilbert basis $\left\{u_{j,\lambda}\right\}$ of eigenfunctions of $H^2_{\lambda,N}(\Omega)$. We note that $p\ne 0$ is an eigenvalue of $T_{\lambda,N}$ if and only if $\mu=\frac{1}{p}-b$ is an eigenvalue of \eqref{weak_BSL}, the eigenfunction being the same.

Formula \eqref{minmaxL} follows from the standard min-max formula for the eigenvalues of compact self-adjoint operators. 

If $\lambda<0$, then $\mu_1(\lambda)=0$ and a corresponding eigenfunction $u_{1,\lambda}$ satisfies $D^2u_{1,\lambda}=0$ in $\Omega$, hence it is a linear function; moreover, since $\frac{\partial u_{1,\lambda}}{\partial\nu}=0$ on $\partial\Omega$, $u_{1,\lambda}$ has to be constant. 

If $\lambda=0$, then $\mu=0$ is an eigenvalue and a corresponding eigenfunction is a linear function. Hence $\mu_1(0)=\cdots=\mu_{N+1}(0)=0$ and the associated eigenspace is spanned by $\left\{1,x_1,...,x_N\right\}$. 

If $0<\lambda<\eta_1$, then by \eqref{limit_eta} and Lemma \ref{limit_mu_0}, there exists $\mu<0$ such that $\lambda_1(\mu)=\lambda$, hence $\mu$ is an eigenvalue of \eqref{weak_BSL}.
%$$
%M_{\lambda}:=\max\left\{\mu\in]-\infty,0[:\lambda_1(\mu)=\lambda\right\}.
%$$
%Such maximum exists (see Theorem \ref{thm-limit-eta} and Lemma \ref{limit_mu_0}).
Moreover, by definition we have that for all $u\in H^2(\Omega)$ with $\gamma_0(u)\ne 0$
$$
\frac{Q_{\lambda,N}(u,u)}{\int_{\partial\Omega}u^2d\sigma}=\frac{Q_{\lambda_1(\mu),N}(u,u)}{\int_{\partial\Omega}u^2d\sigma}\geq \mu,
$$
hence $\mu_1(\lambda)=\mu<0$.

To prove the final part of the theorem, we recast problem \eqref{weak_BSL} into an eigenvalue problem for the compact self-adjoint operator $T_{\lambda,N}'=\gamma_0\circ B_{\lambda,N}^{(-1)}\circ J_0'$, where $J_0'$ denotes the map from $L^2(\partial\Omega)$ to the dual of $H^2(\Omega)$ defined by
$$
J_0'(u)[\varphi]=(u,\gamma_0(\varphi))_{\partial\Omega}\,,\ \ \ \forall u\in L^2(\partial\Omega),\varphi\in H^2(\Omega).
$$
We apply again the Hilbert-Schmidt Theorem and observe that $T_{\lambda,N}$ and $T_{\lambda,N}'$ admit the same non-zero eigenvalues and that the eigenfunctions of $T_{\lambda,N}'$ are exactly the traces of the eigenfunctions of $T_{\lambda,N}$. From \eqref{weak_BSL} we deduce that if we normalize the eigenfunction $u_{j,\lambda}$ of $T_{\lambda,N}$ associated with positive eigenvalues and we denote them by $\tilde u_{j,\lambda}$, then the normalization of their traces in $L^2(\partial\Omega)$ are obtained by multiplying $\gamma_0(\tilde u_{j,\lambda})$ by $\sqrt{\mu_j(\lambda)}$. The rest of the proof easily follows.
\end{proof}

As we have done for problem \eqref{weak_BSM}, we present now a few results on the behavior of the eigenvalues of \eqref{weak_BSL}  for $\lambda\in(-\infty,\eta_1)$. We have the following theorem on the Lipschitz continuity of eigenvalues, the proof of which is similar to that of Theorem \ref{Lipschitz-continuity-M} and is accordingly omitted.

\begin{thm}%\label{Lipschitz-continuity-L}
For any $j\in\mathbb N$ and $\delta>0$, the functions $\mu_j:(-\infty,\eta_1-\delta]\rightarrow (0,+\infty)$ which takes $\lambda\in(-\infty,\eta_1-\delta]$ to $\mu_j(\lambda)\in\mathbb R$ are Lipschitz continuous on $(-\infty,\eta_1-\delta]$.
\end{thm}

We now investigate the behavior of the eigenvalues $\mu_j(\lambda)$ as $\lambda\rightarrow-\infty$. In order state the analogue of Theorem \ref{thm-limit-eta}, we consider the operator $T_N:\mathcal{H}^2_{0,N}(\Omega)\rightarrow \mathcal{H}^2_{0,N}(\Omega)$ defined by $T_N=B_N^{(-1)}\circ J_0$, where $B_N$ is the operator from $\mathcal{H}^2_{0,N}(\Omega)$ to its dual $(\mathcal{H}^2_{0,N}(\Omega))'$ given by
\begin{equation}\label{tB}
B_N(v)[\varphi]=\mathcal Q_{\sigma}(v,\varphi)+b(\gamma_0(v),\gamma_0(\varphi))_{\partial\Omega}\,,\ \ \ \forall v,\varphi\in \mathcal{H}^2_{0,N}(\Omega),
\end{equation}
and $b$ has the same value as in the definition of the operator $T_{\lambda,N}$, see \eqref{constraint-b}, and $J_0$ is defined in \eqref{J0}. Note that the constant $b$ can be chosen to be independent of $\lambda$ for $\lambda<0$. By the Riesz Theorem it follows that $B_N$ is a surjective isometry. The operator $T_N$ is the resolvent operator associated with problem \eqref{weak_NBS} and plays the same role of $T_{\lambda,N}$ defined in \eqref{Tl}. In fact, as in the proof of Theorem \ref{main-BSL} it is possible to show that $T_N$ admits an increasing sequence of non-zero eigenvalues $\left\{p_j\right\}_{j=1}^{\infty}$ bounded from above and converging to $0$. Moreover, a number $p\ne 0$ is an eigenvalue of $T_N$ if and only if $\xi=\frac{1}{p}-b$ is an eigenvalue of \eqref{weak_NBS}, with the same eigenfunctions. 

We have now a family of compact self-adjoint operators $T_{\lambda,N}$ each defined on the Hilbert space $H^2(\Omega)$ endowed with the scalar product \eqref{scalar_L}, and the compact self-adjoint operator $T_N$ defined on $\mathcal{H}^2_{0,N}(\Omega)$ endowed with the scalar product defined by the right-hand side of \eqref{tB}. We have the following theorem, the proof of which is similar to that of Theorem \ref{thm-limit-eta} and is accordingly omitted.

\begin{thm}\label{thm-limit-xi}
The family of operators $\{T_{\lambda,N}\}_{\lambda\in(-\infty,\eta_1)}$ compactly converges to $T_N$ as $\lambda\rightarrow-\infty$. In particular, 
\begin{equation}\label{limit_xi}
\lim_{\lambda\rightarrow-\infty}\mu_j(\lambda)=\xi_j,
\end{equation}
for all $j\in\mathbb N$, where $\xi_j$ are the eigenvalues of \eqref{weak_NBS}.
\end{thm}

%As for the behavior of the eigenvalues as $\lambda\rightarrow-\infty$. We have the following.

%\begin{thm}\label{thm-limit-xi}
%For all $j\in\mathbb N$, 
%\begin{equation*}%\label{limit_xi}
%\lim_{\lambda\rightarrow-\infty}\mu_j(\lambda)=\xi_j,
%\end{equation*}
%where $\xi_j$ are the eigenvalues of \eqref{weak_NBS}.
%\end{thm}
%\begin{proof}
%The proof is identical to that of Theorem \ref{thm-limit-eta} and is accordingly omitted.
%\end{proof}

\begin{rem}
We also note that each eigenvalue $\mu_j(\lambda)$ is non-increasing with respect to $\lambda$, for $\lambda\in(-\infty,\eta_1)$. In fact from the Min-Max Principle \eqref{minmaxL} it immediately follows that for all $j\in\mathbb N$, $\mu_j(\lambda_1)\leq\mu_j(\lambda_2)$ if $\lambda_1>\lambda_2$. 
\end{rem}

%%%%%%%%%%%%%%%%%%%%%%%%%%%%%%%%%%%%%%%%%%%%%%%%%%%%%%%%%%%%%%%%%%%%%%%%%%%%%%%%%%%%%%%%%%%%%%%%%%%%%%%%%%%%%%%%%%%%%%%%%%%%%%%%%%%%%%%%%%%%%%%%%%%%%%%%%%%
%%%%%%%%%%%%%%%%%%%%%%%%%%%%%%%%%%%%%%%%%%%%%%%%%%%%%%%%%%%%%%%%%%%%%%%%%%%%%%%%%%%%%%%%%%%%%%%%%%%%%%%%%%%%%%%%%%%%%%%%%%%%%%%%%%%%%%%%%%%%%%%%%%%%%%%%%%%

%%%%%%%%%%%%%%%%%%%%%%%%%%%%%%%%%%%%%%%%%%%%%%%%%%%%%  %            CHARACTERIZATION OF TRACE SPACE            %%%%%%%%%%%%%%%%%%%%%%%%%%%%%%%%%%%%%%%%%%%%

%%%%%%%%%%%%%%%%%%%%%%%%%%%%%%%%%%%%%%%%%%%%%%%%%%%%%%%%%%%%%%%%%%%%%%%%%%%%%%%%%%%%%%%%%%%%%%%%%%%%%%%%%%%%%%%%%%%%%%%%%%%%%%%%%%%%%%%%%%%%%%%%%%%%%%%%%%%
%%%%%%%%%%%%%%%%%%%%%%%%%%%%%%%%%%%%%%%%%%%%%%%%%%%%%%%%%%%%%%%%%%%%%%%%%%%%%%%%%%%%%%%%%%%%%%%%%%%%%%%%%%%%%%%%%%%%%%%%%%%%%%%%%%%%%%%%%%%%%%%%%%%%%%%%%%%

\section{Characterization of trace spaces of $H^2(\Omega)$ via biharmonic Steklov eigenvalues}\label{traceproblem}

In this section we shall use the Hilbert basis of eigenfunctions $v_{j,\mu}$ and $\hat v_{j,\mu}$ given by Theorem \ref{main-BSM} and the Hilbert basis of eigenfunctions $u_{j,\lambda},\hat u_{j,\lambda}$ given by Theorem \ref{main-BSL}, for all $\mu\in(-\infty,0)$ and $\lambda\in(-\infty,\eta_1)$. We recall that by definition, the functions $v_{j,\mu}$ and $u_{j,\lambda}$ are normalized with respect to $Q_{\mu,D}(\cdot,\cdot)$ and $Q_{\lambda,N}(\cdot,\cdot)+b(\gamma_0(\cdot),\gamma_0(\cdot))_{\partial\Omega}$ respectively, while $\hat v_{j,\mu}$ and $\hat u_{j,\lambda}$ are normalized with respect to the standard scalar product of $L^2(\partial\Omega)$.

We will also denote by $l^2$ the space of sequences $s=(s_j)_{j=1}^{\infty}$ of real numbers satisfying $\|s\|_{l^2}^2=\sum_{j=1}^{\infty}s_j^2<\infty$.

We define the spaces
\begin{equation}\label{S_lambda}
\mathcal S^{\frac{3}{2}}(\partial\Omega)=\mathcal S^{\frac{3}{2}}_{\lambda}(\partial\Omega):=\left\{f\in L^2(\partial\Omega): f=\sum_{j=1}^{\infty}\hat a_j\hat u_{j,\lambda}{\rm\ with\ }\left(\sqrt{|\mu_j(\lambda)|}\hat a_j\right)_{j=1}^{\infty}\in l^2\right\},
\end{equation}
and
\begin{equation}\label{S_mu}
\mathcal S^{\frac{1}{2}}(\partial\Omega)=\mathcal S^{\frac{1}{2}}_{\mu}(\partial\Omega):=\left\{f\in L^2(\partial\Omega): f=\sum_{j=1}^{\infty}\hat b_j\hat v_{j,\mu}{\rm\ with\ }\left(\sqrt{\lambda_j(\mu)}\hat b_j\right)_{j=1}^{\infty}\in l^2\right\}.
\end{equation}
These spaces are endowed with the natural norms defined by
$$
\|f\|_{\mathcal S^{\frac{3}{2}}_{\lambda}(\partial\Omega)}^2=\sum_{j=1}^{j_0-1}\hat a_j^2+\sum_{j=j_{0}}^{\infty}\mu_j(\lambda)\hat a_j^2,
$$
where $j_0$ is as in Theorem \ref{main-BSL}, and
$$
\|f\|_{\mathcal S^{\frac{1}{2}}_{\mu}(\partial\Omega)}^2=\sum_{j=1}^{\infty}\lambda_j(\mu)\hat b_j^2.
$$
Recall that if $\lambda=0$, $j_0=N+2$ and if $\lambda<0$, then $j_0=2$.
%$$
%\|f\|_{\mathcal S^{\frac{3}{2}}_{\lambda}(\partial\Omega)}^2=\sum_{j=1}^{N+1}\hat a_j^2+\sum_{j=N+2}^{\infty}\mu_j(\lambda)\hat a_j^2,
%$$
%and if $\lambda<0$
%$$
%\|f\|_{\mathcal S^{\frac{3}{2}}_{\lambda}(\partial\Omega)}^2=\hat a_1^2+\sum_{j=2}^{\infty}\mu_j(\lambda)\hat a_j^2.
%$$

These spaces allow to describe the trace spaces for $H^2(\Omega)$. In particular, $\mathcal S^{\frac{3}{2}}_{\lambda}(\partial\Omega)$ and $\mathcal S^{\frac{1}{2}}_{\mu}(\partial\Omega)$ turn out to be independent of $\lambda$ and $\mu$. Namely, we have the following.

\begin{thm}\label{traces}
Let $\Omega$ be a bounded domain in $\mathbb R^N$ of class $C^{0,1}$. Then
\begin{equation}\label{trace_0}
\gamma_0(H^2(\Omega))=\gamma_0(H^2_{\lambda,N}(\Omega))=\mathcal S^{\frac{3}{2}}(\partial\Omega)\ (=\mathcal S^{\frac{3}{2}}_{\lambda}(\partial\Omega))
\end{equation}
and
\begin{equation}\label{trace_1}
\gamma_1(H^2(\Omega))=\gamma_1(H^2_{\mu,D}(\Omega))=\mathcal S^{\frac{1}{2}}(\partial\Omega)\ (=\mathcal S^{\frac{1}{2}}_{\mu}(\partial\Omega)).
\end{equation}
In particular, the spaces $\mathcal S_{\lambda}^{\frac{3}{2}}(\partial\Omega)$ and $\mathcal S_{\mu}^{\frac{1}{2}}(\partial\Omega)$ do not depend on $\lambda\in(-\infty,\eta_1)$ and $\mu\in(-\infty,0)$. 

%In particular
%\begin{equation}\label{trace_lipschitz}
%\mathcal S^{\frac{3}{2}}_{\lambda}(\partial\Omega)=H^{\frac{3}{2}}(\Omega)
%\end{equation}
%for all $\lambda\in]-\infty,\eta_1[$.
Moreover, if $\Omega$ is of class $C^{2,1}$ then 
$$
\Gamma(H^2(\Omega))=\mathcal S^{\frac{3}{2}}(\partial\Omega)\times\mathcal S^{\frac{1}{2}}(\partial\Omega),
$$
hence 
\begin{equation*}%\label{trace_lipschitz}
\mathcal S^{\frac{3}{2}}(\partial\Omega)=H^{\frac{3}{2}}(\partial\Omega)
\end{equation*}
and
\begin{equation*}
\mathcal S^{\frac{1}{2}}(\partial\Omega)=H^{\frac{1}{2}}(\partial\Omega).
\end{equation*}

\end{thm}
\begin{proof}%[Proof of Theorem \ref{traces}]
Let us begin by proving \eqref{trace_0}. By the definition of $H^2_{\lambda,N}(\Omega)$ given in \eqref{HL} and by Theorem \ref{main-BSL} we have that any $u\in H^2(\Omega)$ can be written as
$$
u=u_{\lambda}+v_D
$$ 
where $v_D\in \mathcal{H}^2_{0,D}(\Omega)$ and
$$
u_{\lambda}=\sum_{j=1}^{\infty} a_j u_{j,\lambda}
$$
for some coefficients $a_j$ satisfying $\sum_{j=1}^{\infty}a_j^2<\infty$. Here $\left\{u_{j,\lambda}\right\}_{j=1}^{\infty}$ is a orthonormal basis of $H^2_{\lambda,N}(\Omega)$ with respect to the scalar product \eqref{scalar_L} with $b$ satisfying \eqref{constraint-b}. Let $j_0$ be as in Theorem \ref{main-BSL}. Hence we can write
\begin{multline*}
u_{\lambda}=\sum_{j=1}^{j_0-1} a_j u_{j,\lambda}+\sum_{j=j_0}^{\infty} a_j u_{j,\lambda}\\
=\sum_{j=1}^{j_0-1} a_j u_{j,\lambda}+\sum_{j=j_0}^{\infty} \left(a_j \sqrt{\mathcal Q_{\lambda,N}(u_{j,\lambda},u_{j,\lambda})}\right)\cdot \frac{u_{j,\lambda}}{\sqrt{\mathcal Q_{\lambda,N}(u_{j,\lambda},u_{j,\lambda})}}\\
=\sum_{j=1}^{j_0-1} a_j u_{j,\lambda}+\sum_{j=j_0}^{\infty} \tilde a_j \tilde u_{j,\lambda},
\end{multline*}
where $\tilde u_{j,\lambda}=\frac{u_{j,\lambda}}{\sqrt{\mathcal Q_{\lambda,N}(u_{j,\lambda},u_{j,\lambda})}}$ are the eigenfunctions normalized with respect to $\mathcal Q_{\lambda,N}$ and $\tilde a_j$ still satisfy $\sum_{j=j_0+1}^{\infty}\tilde a_j^2 <\infty$ (in fact $0<\mathcal Q_{\lambda,N}(u_{j,\lambda},u_{j,\lambda})\leq 1$ for all $j\geq j_0$).\\
Clearly $\gamma_0(u)=\gamma_0(u_\lambda)$, hence by the continuity of the trace operator we have that
\begin{multline*}
\gamma_0(u_{\lambda})=\sum_{j=1}^{j_0-1} a_j \gamma_0 (u_{j,\lambda})+\sum_{j=j_0}^{\infty} \tilde a_j \gamma_0(\tilde u_{j,\lambda})\\
=\sum_{j=1}^{j_0-1} \frac{a_j}{\sqrt{\mu_j(\lambda)+b}}\cdot \left(\sqrt{\mu_j(\lambda)+b}\cdot\gamma_0 (u_{j,\lambda})\right)+\sum_{j=j_0}^{\infty} \frac{\tilde a_j}{\sqrt{\mu_j(\lambda)}}\cdot \gamma_0\left(\sqrt{\mu_j(\lambda)}\tilde u_{j,\lambda}\right)\\
=\sum_{j=1}^{j_0-1}\hat a_j\hat u_{j,\lambda}+\sum_{j=j_0}^{\infty}\hat a_j\hat u_{j,\lambda}=\sum_{j=1}^{\infty}\hat a_j\hat u_{j,\lambda},
\end{multline*}
where we have set
$$
\hat a_j=\frac{a_j}{\sqrt{\mu_j(\lambda)+b}}\,,\ \ \ \hat u_{j,\lambda}=\sqrt{\mu_j(\lambda)+b}\cdot\gamma_0 (u_{j,\lambda})
$$
for $j=1,...,j_0-1$ and
$$
\hat a_j=\frac{\tilde a_j}{\sqrt{\mu_j(\lambda)}}\,,\ \ \ \hat u_{j,\lambda}=\sqrt{\mu_j(\lambda)}\cdot\gamma_0 (\tilde u_{j,\lambda})
$$
for $j\geq j_0$.
%This follows immediately from the fact that $\mathcal Q_{\lambda,N}(v_{i,\lambda},v_{j,\lambda})+b(\gamma_0(v_{i,\lambda}),\gamma_0(v_{j,\lambda}))=\delta_{i,j}=(b+\mu_j(\lambda))(\gamma_0(v_{i,\lambda}),\gamma_0(v_{j,\lambda}))$ for all $i,j\in\mathbb N$ and $b=M_{\lambda}+1$ (see \eqref{scalar_L}), in particular it is true for all $i,j=1,...,j_0$. 
This proves that $\gamma_0(H^2_{\lambda,N}(\Omega))\subseteq \mathcal S_{\lambda}^{\frac{3}{2}}(\partial\Omega)$.

We prove now the opposite inclusion. Let $f\in\mathcal S_{\lambda}^{\frac{3}{2}}(\partial\Omega)$. Then $f=\sum_{j=1}^{\infty}\hat a_j\hat u_{j,\lambda}$ with $\sum_{j=1}^{\infty}|\mu_j(\lambda)|\hat a_j^2<\infty$. Let $u:=\sum_{j=1}^{\infty}a_j u_{j,\lambda}$ where
\begin{equation}\label{aj}
a_j=\sqrt{\mu_j(\lambda)+b}\cdot\hat a_j
\end{equation}
By definition, $u\in H^2(\Omega)$ since $\sum_{j=1}^{\infty}a_j^2<\infty$. Moreover, we note that
\begin{equation}\label{f!}
f=\sum_{j=1}^{\infty}\hat a_j\hat u_{j,\lambda}=\sum_{j=1}^{\infty}\hat a_j\sqrt{\mu_j(\lambda)+b}\cdot\frac{\hat u_{j,\lambda}}{\sqrt{\mu_j(\lambda)+b}}=\sum_{j=1}^{\infty}\hat a_j\sqrt{\mu_j(\lambda)+b} \cdot \gamma_0(u_{j,\lambda}),
\end{equation}
hence $f=\gamma_0(u)\in\gamma_0(H^2(\Omega))$.

\begin{comment}
, where
\begin{equation}\label{aj_1}
a_j=\sqrt{\mu_j(\lambda)+b}\hat a_j
\end{equation}
for  $j=1,...,j_0$, and
\begin{equation}\label{aj_2}
a_j=\sqrt{\frac{\mu_j(\lambda)}{\mathcal Q_{\lambda,N}(u_{j,\lambda},u_{j,\lambda})}}\hat a_j
\end{equation}
for $j\geq j_0+1$.

Clearly $u$ is a well-defined element in $H^2_{\lambda,N}(\Omega)$ because $\sum_{j=1}^{\infty}a_j^2<\infty$: in fact for all $j\geq j_0+1$
$$
|a_j|\leq\frac{\sqrt{\mu_j(\lambda)}|\hat a_j|}{\mu_{j_0+1}(\lambda)\|\gamma_0(u_{j_0+1,\lambda})\|_{L^2(\partial\Omega)}} 
$$
and the series $|\mu_j(\lambda)|\hat a_j^2$ converges by hypothesis. By the continuity of the trace operator and the definition of $(a_j)_{j=1}^{\infty}$ one immediately verifies that
$$
\gamma_0(u)=\sum_{j=1}^{\infty}a_0\gamma_0(u_{j,\lambda})=\sum_{j=1}^{\infty}\hat a_j\hat u_{j,\lambda}=f,
$$
hence $f\in\gamma_0(H^2_{\lambda,N}(\Omega))$.
\end{comment}
The proof of \eqref{trace_1} follows the same lines as that of \eqref{trace_0} and is accordingly omitted.

We deduce then that the spaces $\mathcal S_{\lambda}^{\frac{3}{2}}(\partial\Omega)$ and $\mathcal S_{\mu}^{\frac{1}{2}}(\partial\Omega)$ do not depend on the particular choice of $\lambda\in(-\infty,\eta_1)$ and $\mu\in(-\infty,0)$. In particular, we have proved that $\Gamma(H^2(\Omega))\subseteq\mathcal S_{\lambda}^{\frac{3}{2}}(\partial\Omega)\times\mathcal S_{\mu}^{\frac{1}{2}}(\partial\Omega)$. %In particular $\mathcal S_{\lambda}^{\frac{3}{2}}(\partial\Omega)=\gamma_0(H^2(\Omega))$.

Assume now that $\Omega$ is of class $C^{2,1}$. We prove that $\mathcal S_{\lambda}^{\frac{3}{2}}(\partial\Omega)\times\mathcal S_{\mu}^{\frac{1}{2}}(\partial\Omega)\subseteq \Gamma(H^2(\Omega))$. This will imply $\Gamma(H^2(\Omega))=\mathcal S_{\lambda}^{\frac{3}{2}}(\partial\Omega)\times\mathcal S_{\mu}^{\frac{1}{2}}(\partial\Omega)$.\\
Let $(f,g)\in \mathcal S_{\lambda}^{\frac{3}{2}}(\partial\Omega)\times \mathcal S_{\mu}^{\frac{1}{2}}(\partial\Omega)$. This means that $f=\gamma_0(u_{\lambda})$, $g=\gamma_1(v_{\mu})$ for some $u_{\lambda}\in H^2_{\lambda,N}(\Omega)$, $v_{\mu}\in H^2_{\mu,D}(\Omega)$. We claim that there exist $v_D\in \mathcal{H}^2_{0,D}(\Omega)$ and $u_N\in \mathcal{H}^2_{0,N}(\Omega)$ such that $u_{\lambda}+v_D=v_{\mu}+u_N$. To do so, it suffices to prove the existence of $v_D\in \mathcal H^2_{0,D}(\Omega)$ and $u_N\in\mathcal H^2_{0,N}(\Omega)$ such that $u_{\lambda}-v_{\mu}=u_N-v_D$. We claim that 
\begin{equation}\label{split}
H^2(\Omega)=\mathcal H^2_{0,D}(\Omega)+\mathcal H^2_{0,N}(\Omega).
\end{equation}
Indeed, given $u\in H^2(\Omega)$, one can find by the classical Total Trace Theorem a function $u_1\in H^2(\Omega)$ such that  $\gamma_0(u_1)=0$ and $\gamma_1(u_1)=\gamma_1(u)$. Thus $u=u_1+(u-u_1)$ with $\gamma_1(u-u_1)=0$ and the claim is proved. Thus the existence of functions $v_D$ and $u_N$ follows by \eqref{split} and
%can be written as a sum of a function in $\mathcal{H}^2_{0,D}(\Omega)$ and a function in $\mathcal{H}^2_{0,N}(\Omega)$ whenever $\Omega$ is of class $C^{1,1}$.
  the function $u=u_{\lambda}+v_D=v_{\mu}+u_N$ is such that $f=\gamma_0(u)$ and $g=\gamma_1(u)$.
\begin{comment}
Assume now that $\Omega$ is of class $C^{2,1}$. It is known that any $u\in H^2(\Omega)$ can be written as $u=v_D+u_N$ with $v_D\in \mathcal{H}^2_{0,D}(\Omega)$ and $u_N\in \mathcal{H}^2_{0,N}(\Omega)$. Let $(f,g)\in \mathcal S_{\lambda}^{\frac{3}{2}}(\partial\Omega)\times \mathcal S_{\mu}^{\frac{1}{2}}(\partial\Omega)$. This means that $f=\gamma_0(u_{\lambda})$, $g=\gamma_1(v_{\mu})$ for some $u_{\lambda}\in\mathcal S^{\frac{3}{2}}_{\lambda}(\partial\Omega)$, $v_{\mu}\in\mathcal S^{\frac{1}{2}}_{\mu}(\partial\Omega)$. Now, $(f,g)\in \Gamma(H^2(\Omega))$ if  $f=\gamma_0(u)$ and $g=\gamma_1(u)$. Hence $u=u_{\lambda}+v_D=v_{\mu}+u_N$ which is possible if and only if there exist $v_D\in \mathcal{H}^2_{0,D}(\Omega)$ and $u_N\in \mathcal{H}^2_{0,N}(\Omega)$ such that $u_{\lambda}-v_{\mu}=u_N-u_D$. This is clearly true since $\Omega$ is of class $C^{2,1}$. Therefore we conclude that
$$
\Gamma(H^2(\Omega))=\mathcal S_{\lambda}^{\frac{3}{2}}(\partial\Omega)\times \mathcal S_{\mu}^{\frac{1}{2}}(\partial\Omega)=H^{\frac{3}{2}}(\partial\Omega)\times H^{\frac{1}{2}}(\partial\Omega).
$$
This implies that if $\Omega$ is of class $C^{2,1}$, 
$$
\mathcal S_{\mu}^{\frac{1}{2}}(\partial\Omega)=H^{\frac{1}{2}}(\partial\Omega)
$$
\end{comment}
\end{proof}

\begin{rem}
%Let $\Omega$ be a bounded domain in $\mathbb R^N$ of class $C^{0,1}$. If we define
%In particular
%\begin{equation}\label{trace_lipschitz}
%\mathcal S^{\frac{3}{2}}_{\lambda}(\partial\Omega)=H^{\frac{3}{2}}(\Omega)
%\end{equation}
%for all $\lambda\in]-\infty,\eta_1[$.
Theorem \ref{traces} gives an explicit spectral characterization of the space $\gamma_0(H^2(\Omega))$ of traces of functions in $H^2(\Omega)$ when $\Omega$ is a bounded domain of class $C^{0,1}$ in $\mathbb R^N$. This space corresponds to $H^{\frac{3}{2}}(\partial\Omega)$ when $\Omega$ is of class $C^{2,1}$. In this case explicit descriptions of the space $H^\frac{3}{2}(\partial\Omega)$ are available in the literature and typically are given by means local charts and explicit representation of derivatives, see e.g., \cite{grisvard,necas}. %Up to our knowledge, a simple explicit description of such space was unavailable up to now. 

For domains of class $C^{0,1}$, it is not clear what is the appropriate definition of $H^{\frac{3}{2}}(\partial\Omega)$. Sometimes $H^{\frac{3}{2}}(\partial\Omega)$ is defined just by setting
$$
H^{\frac{3}{2}}(\partial\Omega):=\gamma_0(H^2(\Omega)).
$$
According to this definition, Theorem \ref{traces} implies that $H^{\frac{3}{2}}(\partial\Omega)=\mathcal S^{\frac{3}{2}}(\partial\Omega)$  also for domains of class $C^{0,1}$.
 
\end{rem}
%\begin{rem}
%By Theorem \ref{traces} we deduce that when $\Omega$ is of class $C^{1,1}$, the standard Sobolev space $H^{\frac{3}{2}}(\partial\Omega)$ can be identified with a subspace of $l^2$ of sequences enjoying some summability property. Moreover each function in the space $H^{\frac{3}{2}}(\partial\Omega)$ can be written as a Fourier series in terms of traces of eigenfunctions of \eqref{weak_BSL} for $\lambda\in(-\infty,\eta_1)$.
%\end{rem}

From Theorem \ref{traces} it follows that if $\Omega$ is a domain of class $C^{0,1}$, then 
\begin{equation}\label{trace_inclusion}
\Gamma(H^2(\Omega))\subseteq \mathcal S^{\frac{3}{2}}(\partial\Omega)\times\mathcal S^{\frac{1}{2}}(\partial\Omega),
\end{equation}
and equality holds if $\Omega$ is of class $C^{2,1}$. We observe that if $\Omega$ is not of class $C^{2,1}$, then in general equality does not hold in \eqref{trace_inclusion}. Indeed, we have the following counterexample.

\begin{cex}\label{cex}
Let $\Omega=(0,1)\times(0,1)$ be unit square in $\mathbb R^2$. We prove that
$$
\Gamma(H^2(\Omega))\subsetneq \mathcal S^{\frac{3}{2}}(\partial\Omega)\times\mathcal S^{\frac{1}{2}}(\partial\Omega).
$$
To do so, we consider the real-valued function $\varphi$ defined in $\Omega$ by $\varphi(x_1,x_2)=x_1$ for all $(x_1,x_2)\in\Omega$ and we prove that the couple $(\gamma_0(\varphi),0)\in(\mathcal S^{\frac{3}{2}}(\partial\Omega)\times\mathcal S^{\frac{1}{2}}(\partial\Omega))\setminus\Gamma(H^2(\Omega))$. It is obvious that $\gamma_0(\varphi)\in \mathcal S^{\frac{3}{2}}(\partial\Omega)$ since $\varphi\in H^2(\Omega)$. Assume now by contradiction that $(\gamma_0(\varphi),0)\in \Gamma(H^2(\Omega))$, that is, there exists $u\in H^2(\Omega)$ such that $\gamma_0(u)=\gamma_0(\varphi)$ and $\gamma_1(u)=0$. Clearly, since $\gamma_0(u)=\gamma_0(\varphi)$, there exists  $v_D\in \mathcal{H}^2_{0,D}$ such that
$$
u=\varphi+v_D
$$
and hence
$$
\gamma_1(v_D)=\gamma_1(u)-\gamma_1(\varphi)=-\nabla x_1\cdot\nu_{|_{\partial\Omega}}=-\nu_1.
$$
It follows that $v_D$ is a function in $H^2(\Omega)$ such that $\gamma_0(v_D)=0$ and $\gamma_1(v_D)=-\nu_1$, but this opposes a well-known necessary (and sufficient) condition for a couple $(f,g)\in H^1(\partial\Omega)\times L^2(\partial\Omega)$ to belong to $\Gamma(H^2(\Omega))$, namely 
\begin{equation}\label{comp_cond}
\frac{\partial f}{\partial \tau}\tau+g\nu\in H^{\frac{1}{2}}(\partial\Omega),
\end{equation}
where $\tau$ is the unit tangent vector (positively oriented with respect to the outer unit $\nu$ to $\Omega$), see \cite{geymonat_airy,grisvard}. Indeed, the couple $(0,-\nu_1)$ does not satisfy condition \eqref{comp_cond}.

% Hence we conclude that if $\Omega$ is not of class $C^{1,1}$, in general
%\begin{equation*}%\label{chain3}
%\mathcal S^{\frac{3}{2}}(\partial\Omega)\times\mathcal S^{\frac{1}{2}}(\partial\Omega)	\subsetneq \Gamma(H^2(\Omega))\subsetneq \mathcal S^{\frac{3}{2}}_{\lambda}(\partial\Omega)\times\mathcal S^{\frac{1}{2}}_{\mu}(\partial\Omega)
%\end{equation*}
\end{cex}

In order to characterize those couples $(f,g)\in \mathcal S^{\frac{3}{2}}(\partial\Omega)\times\mathcal S^{\frac{1}{2}}(\partial\Omega)$ which  belong to $\Gamma(H^2(\Omega))$  when $\Omega$ is of class $C^{0,1}$, we need the spaces $\mathscr S^{\frac{3}{2}}(\partial\Omega)$ and $\mathscr S^{\frac{1}{2}}(\partial\Omega)$ defined by
$$
\mathscr S^{\frac{3}{2}}(\partial\Omega):=\gamma_0(\mathcal{H}^2_{0,N})=\gamma_0(\mathcal B_N(\Omega))
$$
and
$$
\mathscr S^{\frac{1}{2}}(\partial\Omega):=\gamma_1(\mathcal{H}^2_{0,D})=\gamma_1(\mathcal B_D(\Omega)).
$$
The  spaces $\mathscr S^{\frac{3}{2}}(\partial\Omega)$ and $\mathscr S^{\frac{1}{2}}(\partial\Omega)$ 
have  explicit descriptions similar to those of $\mathcal S^{\frac{3}{2}}_{\lambda}(\partial\Omega)$ and $\mathcal S^{\frac{1}{2}}_{\mu}(\partial\Omega)$, namely
\begin{equation}
\label{s32}
\mathscr S ^{\frac{3}{2}}(\partial\Omega )=\biggl\{ f\in L^2(\partial \Omega ):\ f=\sum_{j=1}^{\infty }\hat c_j\hat u_j\  {\rm with}\  ( \sqrt{\xi_j}\hat c_j)_{j=1}^{\infty}\in l^2\biggr\}\, .
\end{equation}
and
\begin{equation}
\label{s12}
\mathscr S ^{\frac{1}{2}}(\partial\Omega)=\biggl\{ g\in L^2(\partial \Omega ):\ g=\sum_{j=1}^{\infty }\hat d_j\hat v_j\  {\rm with}\  (\sqrt{\eta_j}\hat d_j)_{j=1}^{\infty}\in l^2    \biggr\}\, .
\end{equation}
Here $\hat u_j=\sqrt{\xi_j}\gamma_0(u_j)$, $j\geq 2$, where $\left\{u_j\right\}_{j=1}^{\infty}$ is a Hilbert basis of eigenfunctions of problem \eqref{weak_NBS}, normalized with respect to $\mathcal Q_{\sigma}$, with the understanding that $u_1$ and $\hat u_1$ equal the constant  $|\partial\Omega |^{-1/2}$, and $\hat v_j=\sqrt{\eta_j}\gamma_1(v_j)$, where $\left\{v_j\right\}_{j=1}^{\infty}$ is a Hilbert basis of eigenfunctions of problem \eqref{weak_DBS} normalized with respect to $\mathcal Q_{\sigma}$. 

% We shall assume that  $v_n,w_n$ are defined in $\Omega$ and     are normalized with respect to  $Q_{\sigma}$, while $\hat v_n,\hat w_n$ are defined on $\partial \Omega$ and     are normalized with respect to the standard scalar product of $L^2(\partial \Omega )$ ($v_1$, $\hat v_1$ as above).

%On the other hand, we know that 
%$$
%\Gamma(H^2(\Omega))\supseteq\mathcal S^{\frac{3}{2}}(\partial\Omega)\times\mathcal S^{\frac{1}{2}}(\partial\Omega),
%$$
%where
%$$
%\mathcal S^{\frac{3}{2}}(\partial\Omega)=\gamma_0(\mathcal{H}^2_{0,N})
%$$
%and
%$$
%%\mathcal S^{\frac{1}{2}}(\partial\Omega)=\gamma_1(\mathcal{H}^2_{0,D})
%$$

Note that
$$
\mathscr S^{\frac{3}{2}}(\partial\Omega)\times\mathscr S^{\frac{1}{2}}(\partial\Omega)=\Gamma(\mathcal H^2_{0,N}(\Omega)+\mathcal H^2_{0,D}(\Omega))\subseteq \Gamma(H^2(\Omega)).
$$
One can see by similar arguments as in Counterexample \ref{cex} that in general $\mathscr S^{\frac{3}{2}}(\partial\Omega)\times\mathscr S^{\frac{1}{2}}(\partial\Omega)	\subsetneq \Gamma(H^2(\Omega))$ if $\Omega$ is not of class $C^{2,1}$, while equality occurs if $\Omega$ is of class $C^{2,1}$ by \eqref{split}.

We are now ready to characterize the trace space $\Gamma(H^2(\Omega))$ for domains $\Omega$ of class $C^{0,1}$.

\begin{thm}\label{total_trace_2}
Let $\Omega$ be a bounded domain in $\mathbb R^N$ of class $C^{0,1}$. Let $(f,g)\in \mathcal S^{\frac{3}{2}}(\partial\Omega)\times\mathcal S^{\frac{1}{2}}(\partial\Omega)= \mathcal S^{\frac{3}{2}}_{\lambda}(\partial\Omega)\times\mathcal S^{\frac{1}{2}}_{\mu}(\partial\Omega)$ be given by
\begin{equation}\label{total_trace_2_form}
f=\sum_{j=1}^{\infty}\hat a_j \hat u_{j,\lambda}\,,\ \ \ g=\sum_{j=1}^{\infty}\hat b_j \hat v_{j,\mu}
\end{equation}
for some $\lambda\in(-\infty,\eta_1)$, $\mu\in(-\infty,0)$, with $\left(\sqrt{|\mu_j(\lambda)|}\hat a_j\right)_{j=1}^{\infty},\left(\sqrt{\lambda_j(\mu)}\hat b_j\right)_{j=1}^{\infty}\in l^2$. Then $(f,g)$ belongs to $\Gamma(H^2(\Omega))$ if and only if
\begin{equation}\label{comp_1}
\sum_{j=1}^{\infty}a_j\gamma_1(u_{j,\lambda})-g\in \mathscr S^{\frac{1}{2}}(\partial\Omega),
\end{equation}
where $a_j$ are given by \eqref{aj}.

Equivalently, $(f,g)$ belongs to $\Gamma(H^2(\Omega))$ if and only if
\begin{equation}\label{comp_2}
\sum_{j=1}^{\infty}b_j\gamma_0(v_{j,\mu})-f\in \mathscr S^{\frac{3}{2}}(\partial\Omega),
\end{equation}
 where $b_j=\sqrt{\lambda_j(\mu)}\hat b_j$.
\end{thm}
\begin{proof}
Assume that $(f,g)\in \Gamma(H^2(\Omega))$. Then $f=\gamma_0(u_{\lambda}+v_D)$ where $v_D\in \mathcal{H}^2_{0,D}(\Omega)$ and $u_{\lambda}=\sum_{j=1}^{\infty}a_ju_{j,\lambda}$ with the coefficients $a_j$ given by \eqref{aj}. Moreover, $g=\gamma_1(u_{\lambda}+v_D)$ by the continuity of the trace operator. We deduce that
$$
\gamma_1(u_{\lambda})-g=-\gamma_1(v_D)\in\mathscr S^{\frac{1}{2}}(\partial\Omega).
$$
This proves \eqref{comp_1}. Vice versa, assume that \eqref{comp_1} holds. Then there exist $v_D\in \mathcal{H}^2_{0,D}(\Omega)$ such that $\gamma_1(v_D)=\sum_{j=1}^{\infty}a_j\gamma_1(u_{j,\lambda})-g$. Thus 
$$
\gamma_1\left(\sum_{j=1}^{\infty}a_ju_{j,\lambda}-v_D\right)=g
$$
and
$$
\gamma_0\left(\sum_{j=1}^{\infty}a_ju_{j,\lambda}-v_D\right)=f
$$
by \eqref{f!}. The proof of the second part of the statement follows the same lines as that of the first part and is accordingly omitted.
 \end{proof}

\subsection{Representation of the solutions to the Dirichlet problem}\label{dir_prob_2}
Using the Steklov expansions in \eqref{S_lambda} and \eqref{S_mu} and the characterization of the total trace space $\Gamma(H^2(\Omega))$ given by Theorem \ref{total_trace_2} we are able to describe the solutions to the Dirichlet problem \eqref{Dirichlet_problem}.

\begin{cor}
Let $\Omega$ be a bounded domain in $\mathbb R^N$ of class $C^{0,1}$, $(f,g)\in L^2(\partial\Omega)\times L^2(\partial\Omega)$. Then, there exists a solution $u\in H^2(\Omega)$ to problem \eqref{Dirichlet_problem} if and only if the couple $(f,g)$ belongs to $\mathcal S^{\frac{3}{2}}(\partial\Omega)\times \mathcal S^{\frac{1}{2}}(\partial\Omega)$ and satisfies condition \eqref{comp_1} or, equivalently, condition \eqref{comp_2}. In this case, if $f,g$ are represented as in \eqref{total_trace_2_form}, then the solution $u$ can be represented in each of the following two forms:
\begin{enumerate}[i)]
\item if $u_{\lambda}:=\sum_{j=1}^{\infty}a_j u_{j,\lambda}$ where $a_j$ are given by \eqref{aj} and $g-\gamma_1(u_{\lambda})$ is represented by $\sum_{j=1}^{\infty}\hat d_j\hat v_j\in\mathscr S^{\frac{1}{2}}(\partial\Omega)$, then
$$
u=u_{\lambda}+v_D
$$
with $u_{\lambda}\in H^2_{\lambda,N}(\Omega)$ and $v_D=\sum_{j=1}^{\infty}d_jv_j\in\mathcal B_D(\Omega)$, $d_j=\sqrt{\eta_j}\hat d_j$ for all $j\in\mathbb N$.

\item if $v_{\mu}:=\sum_{j=1}^{\infty}b_j v_{j,\mu}$ where $b_j=\sqrt{\lambda_j(\mu)}\hat b_j$ and $f-\gamma_0(v_{\mu})$ is represented by $\sum_{j=1}^{\infty}\hat c_j\hat u_j\in\mathscr S^{\frac{3}{2}}(\partial\Omega)$, then
$$
u=v_{\mu}+u_N
$$
with $v_{\mu}\in H^2_{\mu,D}(\Omega)$ and $u_N=\sum_{j=1}^{\infty}c_ju_j\in\mathcal B_N(\Omega)$, $c_j=\sqrt{\xi_j}\hat c_j$ for all $j\in\mathbb N$, $j\geq 2$, $c_1=\hat c_1$.

%da})-g=-\gamma_1(v_D)\in\mathcal S^{\frac{1}{2}}$, $\gamma_1(v_D)=\sum_{j=1}^{\infty}\hat d_j\hat v_j$ as in \eqref{s12} 
%\item  $u=v_{\mu}+u_N$ with $v_{\mu}=\sum_{j=1}^{\infty}b_j v_{j,\mu}\in H^2_{\mu,D}(\Omega)$,  $\gamma_0(v_{\mu})-f=-\gamma_0(u_N)\in\mathcal S^{\frac{3}{2}}$, $\gamma_0(u_N)=\sum_{j=1}^{\infty}\hat c_j\hat u_j$ as in \eqref{s32} and  $u_N=\sum_{j=1}^{\infty} c_j u_j\in\mathcal B_N(\Omega)$. 
\end{enumerate}

Moreover the solution $u$ is unique
%$$
%u=\sum_{j=1}^{\infty}a_j u_{j,\lambda}+d_jv_j=\sum_{j=1}^{\infty}b_j v_{j,\mu}+c_j u_j,
%$$
%is the unique solution in $H^2(\Omega)$ to the Dirichlet problem \eqref{Dirichlet_problem}.
\end{cor}
\begin{proof}
The first part of the statement is an immediate consequence of Theorem \ref{total_trace_2}. Indeed, if there exists a solution $u\in H^2(\Omega)$ then $(f,g)$ belongs to $\Gamma(H^2(\Omega))$, hence it satisfies \eqref{comp_1} and \eqref{comp_2}. 

Vice versa, if $(f,g)$ satisfies \eqref{comp_1}, then $u_{\lambda}\in H^2_{\lambda,N}(\Omega)$, $v_D\in\mathcal B_D(\Omega)$ are well-defined and $u=u_{\lambda}+v_D$ is a biharmonic function in $H^2(\Omega)$ such that $\gamma_0(u)=\gamma_0(u_{\lambda})=f$ and $\gamma_1(v_D)=g-\gamma_1(u_{\lambda})$, hence $\gamma_1(u)=\gamma_1(u_{\lambda}+v_D)=g$.

Similarly, if $(f,g)$ satisfies \eqref{comp_2}, then $v_{\mu}\in H^2_{\mu,D}(\Omega)$, $u_N\in\mathcal B_N(\Omega)$ are well-defined and $u=v_{\mu}+u_N$ is a biharmonic function in $H^2(\Omega)$ such that $\gamma_1(u)=\gamma_1(v_{\mu})=g$ and $\gamma_0(u_N)=f-\gamma_0(v_{\mu})$, hence $\gamma_0(u)=\gamma_0(v_{\mu}+u_N)=f$.

The uniqueness of the solution  in $H^2(\Omega)$ follows from the fact that a solution $u$ in $H^2(\Omega)$ of \eqref{Dirichlet_problem} with $f=g=0$ must belong to $H^2_0(\Omega)$ and, since it is biharmonic, it must also belong to the orthogonal of $H^2_0(\Omega)$, hence $u=0$.
\end{proof}

\appendix

 %%%%%%%%%%%%%%%%%%%%%%%%%%%%%%%%%%%%%%%%%%%%%%%%%%%%%%%%%%%%%%%%%%%%%%%%%%%%%%%%%%%%%%%%%%%%%%%%%%%%%%%%%%%%%%%%%%%%%%%%%%%%%%%%%%%%%%%%%%%%%%%%%%%%%%%%%%%
%%%%%%%%%%%%%%%%%%%%%%%%%%%%%%%%%%%%%%%%%%%%%%%%%%%%%%%%%%%%%%%%%%%%%%%%%%%%%%%%%%%%%%%%%%%%%%%%%%%%%%%%%%%%%%%%%%%%%%%%%%%%%%%%%%%%%%%%%%%%%%%%%%%%%%%%%%%

%%%%%%%%%%%%%%%%%%%%%%%%%%%%%%%%%%%%%%%%%%%%%%%%%%%%%%%%%            EXPLICIT EXAMPLE - BALL                 %%%%%%%%%%%%%%%%%%%%%%%%%%%%%%%%%%%%%%%%%%%%%%

%%%%%%%%%%%%%%%%%%%%%%%%%%%%%%%%%%%%%%%%%%%%%%%%%%%%%%%%%%%%%%%%%%%%%%%%%%%%%%%%%%%%%%%%%%%%%%%%%%%%%%%%%%%%%%%%%%%%%%%%%%%%%%%%%%%%%%%%%%%%%%%%%%%%%%%%%%%
%%%%%%%%%%%%%%%%%%%%%%%%%%%%%%%%%%%%%%%%%%%%%%%%%%%%%%%%%%%%%%%%%%%%%%%%%%%%%%%%%%%%%%%%%%%%%%%%%%%%%%%%%%%%%%%%%%%%%%%%%%%%%%%%%%%%%%%%%%%%%%%%%%%%%%%%%%%

\section{Eigenvalues and eigenfunctions on the ball}\label{ball}

In this section we compute the eigenvalues and the eigenfunctions of \eqref{BSM} and \eqref{BSL} when $\Omega=B$ is the unit ball in $\mathbb R^N$ centered at the origin and $\sigma=0$. It is convenient to use spherical coordinates $(r,\theta)$, where $\theta=(\theta_1,...,\theta_{N-1})$. The corresponding transformation of coordinates is
\begin{eqnarray*}
x_1&=&r \cos(\theta_1),\\
x_2&=&r\sin(\theta_1)\cos(\theta_2),\\
\vdots\\
x_{N-1}&=&r\sin(\theta_1)\sin(\theta_2)\cdots\sin(\theta_{N-2})\cos(\theta_{N-1}),\\
x_N&=&r\sin(\theta_1)\sin(\theta_2)\cdots\sin(\theta_{N-2})\sin(\theta_{N-1}),
\end{eqnarray*}
with $\theta_1,...,\theta_{N-2}\in [0,\pi]$, $\theta_{N-1}\in [0,2\pi)$ (here it is understood that $\theta_1\in [0,2\pi)$ if $N=2$).

The boundary conditions for fixed parameters $\lambda,\mu\in\mathbb R$
\begin{equation}\label{BC}
\begin{cases}
\frac{\partial^2 u}{\partial\nu^2}=\lambda\frac{\partial u}{\partial\nu}, & {\rm on\ }\partial B,\\
-{\rm div}_{\partial\Omega}(D^2u\cdot\nu)_{\partial\Omega}-\frac{\partial\Delta u}{\partial\nu}=\mu u, & {\rm on\ }\partial B,
\end{cases}
\end{equation}
are written in spherical coordinates as
\begin{equation}\label{BC2}
\begin{cases}
\frac{\partial^2 u}{\partial r^2 }_{|_{r=1}}=\lambda\frac{\partial u}{\partial r }_{|_{r=1}} ,\\
-\frac{1}{r^2}{\Delta_S}\Big(\frac{\partial u}{\partial r}-\frac{u}{r}\Big)-\frac{\partial\Delta u}{\partial r}_{|_{r=1}}=\mu u_{|_{r=1}},
\end{cases}
\end{equation}
where $\Delta_S$ is the Laplace-Beltrami operator on the unit sphere. It is well known that the eigenfunctions can be written as a product of a radial part and an angular part (see e.g., \cite{chasman} for details). In particular, the radial part is given in terms of power-type functions, while the angular part is written in terms of spherical harmonics.  We have the following theorem.

\begin{thm}\label{fundamental-2}
Let $\Omega=B$ be the unit ball in $\mathbb R^N$ and $(\lambda,\mu)\in\mathbb R^2$. Then there exists a non-trivial solution to problem $\Delta^2u=0$ on $B$ with boundary conditions \eqref{BC} if and only if there exists $l\in\mathbb N_0$ such that ${\rm det}M_l(\lambda,\mu)=0$, where $M_l(\lambda,\mu)$ is the matrix defined by
\begin{equation}\label{matrix}
M_l(\lambda,\mu)=
\begin{pmatrix} 
l(l-1-\lambda) & (l+2)(l+1-\lambda) \\
 l(l+N-2)(l-1)-\mu &     l(l(l-5)+N(l-1)-2)-\mu,
\end{pmatrix}
\end{equation}
for all $l\in\mathbb N_0$. If $(\lambda,\mu)\in\mathbb R^2$ solves the equation ${\rm det}M_l(\lambda,\mu)=0$ for some $l\in\mathbb N_0$, then the associated solutions can be written in the form
\begin{equation*}%\label{eigenfunction-2}
u_l(r,\theta)=\left(A_lr^{l}+B_lr^{l+2}\right)H_l(\theta),
\end{equation*}
where $(A_l,B_l)\in\mathbb R^2$ solves the linear system $M_l(\lambda,\mu)\cdot (A_l,B_l)=0$ and $H_l(\theta)$ is a spherical harmonic of degree $l$ in $\mathbb R^N$. 
\end{thm}
\begin{proof}
It is well-known that the weak solutions of $\Delta^2u=0$ in the unit ball complemented with \eqref{BC} are smooth (see e.g., \cite[\S2]{gazzola}). Moreover, we recall that any function $u$ satisfying $\Delta^2u=0$ on the unit ball $B$ along with the two homogeneous boundary conditions \eqref{BC2} can be written in spherical coordinates in the form
\begin{equation}\label{general_ball}
u_l(r,\theta)=(A r^l+B r^{l+2})H_l(\theta),
\end{equation}
for $l\in\mathbb N_0$, where $A,B\in\mathbb R$ are arbitrary constants and $H_l(\theta)$ is a spherical harmonic of degree $l$ in $\mathbb R^N$ (see e.g., \cite[\S5-6]{buosoprovenzano} for details). By using the explicit form \eqref{general_ball} in \eqref{BC2}, we find that the constants $A,B$ need to satisfy a  homogeneous system of two linear equations, whose associated matrix is given by \eqref{matrix}. Hence a non-trivial solution exists if and only if the determinant of \eqref{matrix} is zero. The rest of the statement is now a straightforward consequence.
\end{proof}

By Theorem \ref{fundamental-2} we immediately deduce the following characterization of the eigenvalues of \eqref{BSM} and \eqref{BSL}. Here by $m_l$ we denote the dimension of the space of the spherical harmonics of degree $l\in\mathbb N_0$ in $\mathbb R^N$, that is
$$
m_l=\frac{(2l+N-2)(l+N-3)!}{l!(N-2)!}.
$$

\begin{cor}\label{eigenvalues_ball}
Any eigenvalue $\lambda(\mu)$ of \eqref{BSM} on $B$ satisfies the equation
\begin{multline}\label{ball_LM}
\lambda(\mu)\left(2l^3+(N-1)l^2-(N-2)l-\mu\right)\\
=\left(3l^4+2(N-2)l^3-(N+1)l^2-(N-2)l-(1+2l)\mu\right),
\end{multline}
for some $l\in\mathbb N_0$. Any eigenvalue $\mu(\lambda)$ of \eqref{BSL} on $B$ satisfies the equation
\begin{multline}\label{ball_ML}
\mu(\lambda)\left(2l+1-\lambda\right)\\
=\left(3l^4+2(N-2)l^3-(N+1)l^2-(N-2)l-\left(2l^3+(N-1)l^2-(N-2)l\right)\lambda\right),
\end{multline}
for some $l\in\mathbb N_0$. The multiplicity of $\mu(\lambda)$ and $\lambda(\mu)$ corresponding to an index $l\in\mathbb N_0$ equals the dimension $m_l$ of the space of the spherical harmonics of degree $l$ in $\mathbb R^N$.
\end{cor}
%\begin{proof}
%By exploiting ${\rm det}M_l(\mu,\lambda)=0$, where $M_l(\mu,\lambda)$ is given by \eqref{matrix}, immediately yields that any eigenvalue of \eqref{BSM} satisfies \eqref{ball_LM} and any eigenvalue of \eqref{BSM} satisfies \eqref{ball_ML}.
%\end{proof}

By using similar arguments, one can easily prove that the eigenvalues of problems \eqref{DBS} and \eqref{NBS} on the unit ball $B$ can also be determined explicitly. 

\begin{thm}\label{NBS_DBS}
Any eigenvalue $\eta$ of problem \eqref{DBS} on $B$ is of the form
\begin{equation}\label{ball_eta}
\eta=2l+1
\end{equation}
for some $l\in\mathbb N_0$ and its multiplicity equals the dimension $m_l$ of the space of spherical harmonics of degree $l$ in $\mathbb R^N$.

Any eigenvalue $\xi$ of problem \eqref{NBS} on $B$ is of the form
\begin{equation}\label{ball_xi}
\xi=l(2l^2+(N-1)l-N+2),
\end{equation}
for some $l\in\mathbb N_0$ and its multiplicity equals the dimension $m_l$ of the space of spherical harmonics of degree $l$ in $\mathbb R^N$.
\end{thm}

%We shall change now the notation for the eigenvalues. For $l\in\mathbb N$, we denote by  $\lambda_{(l)}(\mu)$, $\mu_{(l)}(\lambda)$, $\eta_{(l)}$ and $\xi_l$ an eigenvalue of \eqref{BSM}, \eqref{BSL}, \eqref{DBS} and \eqref{NBS}, respectively, corresponding to an index $l\in\mathbb N$. All such eigenvalues have multiplicity $m_l$, and we don't repeat here multiple eigenvalues. 

%With this notation, as a consequence of Theorems \ref{fundamental-2} and \ref{NBS_DBS}, we have explicit representations of the eigenvalues of \eqref{BSM}-\eqref{BSL}.
By combining Corollary \ref{eigenvalues_ball} and Theorem \ref{NBS_DBS} we can prove the following statement, where $\lambda_{(l)}(\mu)$, $\mu_{(l)}(\lambda)$, $\eta_{(l)}$ and $\xi_{(l)}$ denote the eigenvalues of \eqref{BSM}, \eqref{BSL}, \eqref{DBS} and \eqref{NBS}, respectively, associated with spherical harmonics of order $l\in\mathbb N_0$.

\begin{thm}%\label{eigenvalues_ball_2}
For all $l\in\mathbb N$ and $\mu\ne\xi_{(l)}$
\begin{equation}\label{ball_LM_2}
\lambda_{(l)}(\mu)=\frac{\left(3l^4+2(N-2)l^3-(N+1)l^2-(N-2)l-\eta_{(l)}\mu\right)}{\left(\xi_{(l)}-\mu\right)}.
\end{equation}
For all $l\in\mathbb N$ and $\lambda\ne\eta_{(l)}$
\begin{equation}\label{ball_ML_2}
\mu_{(l)}(\lambda)=\frac{\left(3l^4+2(N-2)l^3-(N+1)l^2-(N-2)l-\xi_{(l)}\lambda\right)}{\left(\eta_{(l)}-\lambda\right)}.
\end{equation}
Moreover, $\lambda_{(0)}(\mu)=\eta_{(0)}=1$ for all $\mu\in\mathbb R$ and $\mu_{(0)}(\lambda)=\xi_{(0)}=0$  for all $\lambda\in\mathbb R$.
\end{thm}
\begin{proof}
According to the change of notation for eigenvalues, we  denote by $\lambda_{(l)}(\mu)$, $\mu_{(l)}(\lambda)$, $\eta_{(l)}$ and $\xi_{(l)}$  the eigenvalues of problems \eqref{BSM}, \eqref{BSL}, \eqref{DBS} and \eqref{NBS} corresponding to the choice of $l\in\mathbb N_0$ in \eqref{ball_LM}, \eqref{ball_ML}, \eqref{ball_eta} and \eqref{ball_xi}. Each of such eigenvalues has multiplicity $m_l$. 

We note that \eqref{ball_LM} and \eqref{ball_ML} can be rewritten as
\begin{equation}\label{ball_LM_1}
\lambda_{(l)}(\mu)\left(\xi_{(l)}-\mu\right)=\left(3l^4+2(N-2)l^3-(N+1)l^2-(N-2)l-\eta_{(l)}\mu\right)
\end{equation}
and
\begin{equation}\label{ball_ML_1}
\mu_{(l)}(\lambda)\left(\eta_{(l)}-\lambda\right)=\left(3l^4+2(N-2)l^3-(N+1)l^2-(N-2)l-\xi_{(l)}\lambda\right).
\end{equation}

If $l\in\mathbb N$, then from \eqref{ball_LM_1} and \eqref{ball_ML_1} we deduce the validity of \eqref{ball_LM_2} when $\mu\ne\xi_{(l)}$ and of \eqref{ball_ML_2} when $\lambda\ne \eta_{(l)}$. If $l=0$, the condition ${\rm det}M_0(\lambda,\mu)=0$ can be written in the form $\mu(\lambda-1)=0$ which allows to conclude the proof.%A straightforward computation shows that the right-hand side of both \eqref{ball_LM_2} and \eqref{ball_ML_2} is not a constant for $l\in\mathbb N$, $l\geq 1$.

\begin{comment}
We note now that $\xi_0=0$ and $\eta_0=1$ for all $N\geq 2$. By taking $l=0$ in \eqref{ball_LM_1} and \eqref{ball_ML_1}, we obtain
$$
-\mu\lambda_0(\mu)=-\mu
$$
and
$$
\mu_0(\lambda)(1-\lambda)=0,
$$
which imply that $\lambda_0(\mu)=1$ for all $\mu\in\mathbb R\setminus\left\{0\right\}$ and $\mu_0(\lambda)=0$ for all $\lambda\in\mathbb R\setminus\left\{1\right\}$. Moreover, if $\mu=0$ an eigenvalue $\lambda_0(0)$ is such that
\begin{equation}\label{quad0}
\mathcal Q_{\sigma}(v_0,\varphi)=\lambda_0(0)(\gamma_1(v_0),\gamma_1(\varphi))\,,\ \ \ \forall\varphi\in H^2(B),
\end{equation}
where $v_0$ is an eigenfunction associated to $\lambda_0(0)$. Now, being $l=0$, an eigenfunction associated with $\lambda_0(0)$ is of the form (in spherical coordinates) $v_0(r,\theta)=A_0+B_0r^2$, with $A_0,B_0\in\mathbb R$. Computing $Q_{\sigma}(v_0,\varphi)$ and $(\gamma_1(v_0),\gamma_1(\varphi))_{\partial\Omega}$, we immediately see that \eqref{quad0} becomes
$$
2B_0\int_B\Delta\varphi dx=2B_0\lambda_0(0)\int_{\partial\Omega}\frac{\partial\varphi}{\partial\nu}d\sigma\,,\ \ \ \forall\varphi\in H^2(B),
$$
which, for $B_0\ne 0$, is true if and only if $\lambda_0(0)=1$. We conclude that $\lambda_0(\mu)=1$ for all $\mu\in\mathbb R$. In the same way we see that, when $\lambda=1$, $\mu_0(1)=0$, therefore  $\mu_0(\lambda)=0$ for all $\lambda\in\mathbb R$.

\end{comment}
\end{proof}

We note that
$$
\lim_{\mu\rightarrow-\infty}\lambda_{(l)}(\mu)=\eta_{(l)}
$$
and
$$
\lim_{\lambda\rightarrow-\infty}\mu_{(l)}(\lambda)=\xi_{(l)}
$$
for all $l\in\mathbb N_0$. This is coherent with Theorems \ref{thm-limit-eta} and \ref{thm-limit-xi}. Moreover,
$$
\lim_{\mu\rightarrow\xi_{(l)}^{\pm}}\lambda_{(l)}(\mu)=\pm\infty
$$
and
$$
\lim_{\lambda\rightarrow\eta_{(l)}^{\pm}}\mu_{(l)}(\lambda)=\pm\infty
$$
for all $l\in\mathbb N$.
%  Moreover, {\color{red}(this has to be proved, but if it is too effort for just this sentence, maybe remove it)} for all $\lambda<\eta_1=1$, $\mu_{l+1}(\lambda)>\mu_{l}(\lambda)$, and for all  $\mu<\xi_1=0$, $\lambda_{l+1}(\mu)>\lambda_{l}(\mu)$.
%We note that formulas \eqref{ball_LM_2} and \eqref{ball_ML_2} holds respectively for any $\mu\in]\xi_j,\xi_{j+1}[$ and $\lambda\in]\eta_j,\eta_{j+1}[$ and all $j\in\mathbb N$. It remains to prove what happens in correspondence of $\mu=\xi_j$ and $\lambda=\eta_j$. First we note that there exists particular values of $\mu,\lambda,l$ for which equations \eqref{ball_LM_1} and \eqref{ball_ML_2} are always satisfied.
%In particular, we see that if $\mu=\xi_1=0$ (hence $l=0$), \eqref{ball_LM_1} is always true. 
%The picture is now complete: for all $l\geq 2$
We have shown that the branches of eigenvalues $\lambda_{(l)}(\mu)$ and $\mu_{(l)}(\lambda)$ are analytic functions of their parameters on $\mathbb R\setminus\{\xi_{(l)}\}_{l\in\mathbb N_0}$ and $\mathbb R\setminus\{\eta_{(l)}\}_{l\in\mathbb N_0}$ respectively. In particular, the branch $\lambda_{(l)}(\mu)$ is a equilateral hyperbole with $\eta_{(l)}$ as horizontal asymptote and $\xi_{(l)}$ as vertical asymptote, if $l\geq 1$, while it is coincides with $\left\{(\mu,1):\mu\in\mathbb R\right\}$ for $l=0$. The branch $\mu_{(l)}(\lambda)$ is a equilateral hyperbole with $\xi_{(l)}$ as horizontal asymptote and $\eta_{(l)}$ as vertical asymptote, if $l\geq 1$, while it is coincides with $\left\{(\lambda,0):\mu\in\mathbb R\right\}$ for $l=0$. The situation is illustrated in Figure \ref{fig1}.
\begin{figure}[H]
\centering
\includegraphics[width=\textwidth]{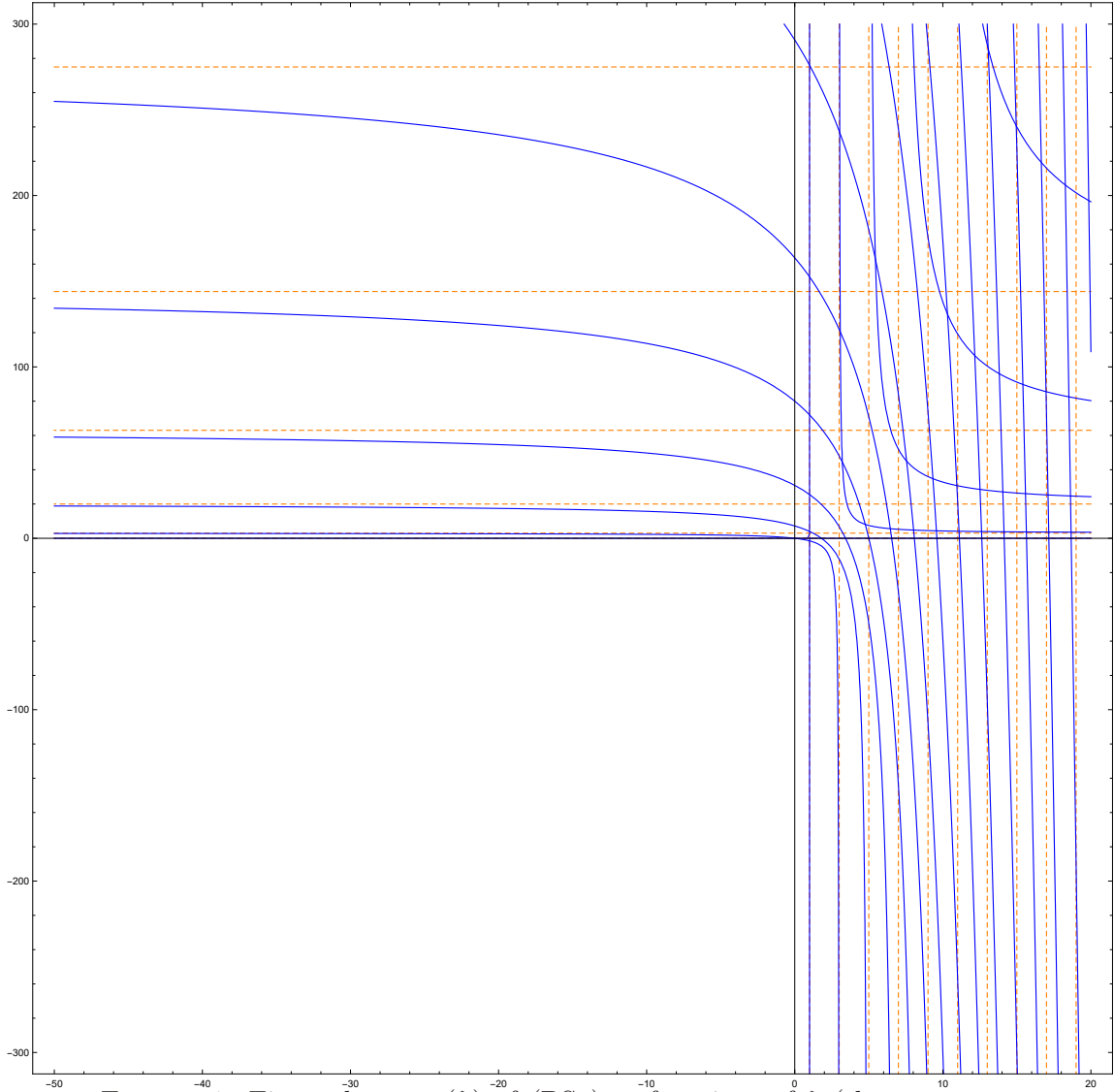}
\caption{Eigenvalues $\mu_{(l)}(\lambda)$ of \eqref{BSL} as functions of $\lambda$ (the parameter $\lambda$ correspond to the abscissa). Vertical asymptotes are the eigenvalues $\eta_{(l)}$ of \eqref{DBS}. Horizontal asymptotes are the eigenvalues $\xi_{(l)}$ of \eqref{NBS}. A reflection along the angle bisector of the first and third quadrant gives the eigenvalues $\lambda_{(l)}(\mu)$ of \eqref{BSM} as functions of $\mu$.}
\label{fig1}
\end{figure}

%%%%%%%%%%%%%%%%%%%%%%%%%%%%%%%%%%%%%%%%%%%%%%%%%%%%%%%%%%%%%%%%%%%%%%%%%%%%%%%%%%%%%%%%%%%%%%%%%%%%%%%%%%%%%%%%%%%%%%%%%%%%%%%%%%%%%%%%%%%%%%%%%%%%%%%%%%%
%%%%%%%%%%%%%%%%%%%%%%%%%%%%%%%%%%%%%%%%%%%%%%%%%%%%%%%%%%%%%%%%%%%%%%%%%%%%%%%%%%%%%%%%%%%%%%%%%%%%%%%%%%%%%%%%%%%%%%%%%%%%%%%%%%%%%%%%%%%%%%%%%%%%%%%%%%%

%%%%%%%%%%%%%%%%%%%%%%%%%%%%%%%%%%%%%%%%%%%%%%%%%%%%%%%%%              ASYMPTOTIC FORMULAS                   %%%%%%%%%%%%%%%%%%%%%%%%%%%%%%%%%%%%%%%%%%%%%%

%%%%%%%%%%%%%%%%%%%%%%%%%%%%%%%%%%%%%%%%%%%%%%%%%%%%%%%%%%%%%%%%%%%%%%%%%%%%%%%%%%%%%%%%%%%%%%%%%%%%%%%%%%%%%%%%%%%%%%%%%%%%%%%%%%%%%%%%%%%%%%%%%%%%%%%%%%%
%%%%%%%%%%%%%%%%%%%%%%%%%%%%%%%%%%%%%%%%%%%%%%%%%%%%%%%%%%%%%%%%%%%%%%%%%%%%%%%%%%%%%%%%%%%%%%%%%%%%%%%%%%%%%%%%%%%%%%%%%%%%%%%%%%%%%%%%%%%%%%%%%%%%%%%%%%%

\section{Asymptotic formulas}\label{asymptotic}

It is proved in \cite{liu_2,liu_1} that if $\Omega$ is a bounded domain in $\mathbb R^N$ with $C^{\infty}$ boundary, then the eigenvalues of problems \eqref{DBS} and \eqref{NBS}  with $\sigma=1$ satisfy the following asymptotic laws 
\begin{equation}\label{weyl_DBS}
\eta_j\sim \frac{4\pi}{\omega_{N-1}^{\frac{1}{N-1}}}\left(\frac{j}{|\partial\Omega|}\right)^{\frac{1}{N-1}},
\end{equation}
and
\begin{equation}\label{weyl_NBS}
\xi_j\sim \frac{16\pi^3}{\omega_{N-1}^{\frac{3}{N-1}}}\left(\frac{j}{|\partial\Omega|}\right)^{\frac{3}{N-1}},
\end{equation}
as $j\rightarrow\infty$. Here $\omega_{N-1}$ denotes the volume of the unit ball in $\mathbb R^{N-1}$. 

We state the following theorem, whose proof is omitted since it follows exactly the same lines as those of Theorems 1.1 and 1.2 of \cite{liu_1}.
\begin{thm}
Let $\Omega$ be a bounded domain in $\mathbb R^N$ of class $C^{\infty}$. Then formulas \eqref{weyl_DBS} and \eqref{weyl_NBS} hold for all $\sigma\in\big(-\frac{1}{N-1},1\big)$. Moreover, the two following asymptotic formulas hold for the eigenvalues of problems \eqref{BSM} and \eqref{BSL}
\begin{equation}\label{weyl_BSM}
\lambda_j(\mu)\sim \frac{3\pi}{\omega_{N-1}^{\frac{1}{N-1}}}\left(\frac{j}{|\partial\Omega|}\right)^{\frac{1}{N-1}},
\end{equation}
and
\begin{equation}\label{weyl_BSL}
\mu_j(\lambda)\sim \frac{12\pi^3}{\omega_{N-1}^{\frac{3}{N-1}}}\left(\frac{j}{|\partial\Omega|}\right)^{\frac{3}{N-1}}
\end{equation}
as $j\rightarrow\infty$.% {\color{red}(they hold for $\lambda<\eta_1$, $\mu<0$ but also for all values such that we have discrete spectrum etc.)}
\end{thm}

We note  that the principal term in the asymptotic expansions of the eigenvalues depends neither on the Poisson's ratio $\sigma$ nor on $\mu$ or $\lambda$. However, lower order terms have to depend on $\mu$ and $\lambda$, since, as $\mu,\lambda\rightarrow-\infty$, $\lambda_j(\mu)\rightarrow\eta_j$ and $\mu_j(\lambda)\rightarrow\xi_j$, and asymptotic formulas of $\lambda_j(\mu)$ and $\eta_j$, and of $\mu_j(\lambda)$ and $\xi_j$, differ from a factor $\frac{3}{4}$.

\begin{rem}
We remark that the approach used in \cite{liu_1} requires that the boundary of $\Omega$ is of class $C^{\infty}$. However, as proved by another technique in \cite{liu_2}, the asymptotic formulas for the eigenvalues of \eqref{DBS} and \eqref{NBS} when $\sigma=1$ hold when $\Omega$ is of class $C^{2}$.
\end{rem}

We now show that in the case of the unit ball $B$ in $\mathbb R^N$ it is possible to recover formulas \eqref{weyl_DBS}, \eqref{weyl_NBS}, \eqref{weyl_BSM} and \eqref{weyl_BSL} by using the explicit computations in Appendix \ref{ball}.

Note that for a fixed $l\in\mathbb N_0$, the dimension of the space of spherical harmonics of degree less or equal than $l$ is $\frac{(2l+N-1)(N+l-2)!}{l!(N-1)!}$. By  \eqref{ball_xi} we deduce that
\begin{equation}\label{asymptotic_xi_k}
\xi_j=l(2l^2+(N-1)l-N+2)
\end{equation}
whenever $j\in\mathbb N$ is such that
\begin{equation}\label{kl_xi}
\frac{(2l+N-3)(N+l-3)!}{(l-1)!(N-1)!}<j\leq\frac{(2l+N-1)(N+l-2)!}{l!(N-1)!}
\end{equation}
Moreover,
\begin{multline*}
\lim_{l\rightarrow+\infty}\frac{(2l+N-3)(N+l-3)!}{(l-1)!(N-1)!}\div\frac{2l^{N-1}}{(N-1)!}\\
=\lim_{l\rightarrow+\infty}\frac{(2l+N-1)(N+l-2)!}{l!(N-1)!}\div\frac{2l^{N-1}}{(N-1)!}=1.
\end{multline*}
From \eqref{asymptotic_xi_k} and \eqref{kl_xi} we deduce that
$$
\xi_{j}\sim 2^{\frac{N-4}{N-1}}(N-1)!^{\frac{3}{N-1}}j^{\frac{3}{N-1}}\,,\ \ \ {\rm as\ }j\rightarrow+\infty.
$$
%In fact, in view of \eqref{kl_xi}, by substituting into \eqref{asymptotic_xi_k} $l=\left(\frac{j(N-1)!}{2}\right)^{\frac{1}{N-1}}$ we obtain
%\begin{multline*}
%\left(\frac{j(N-1)!}{2}\right)^{\frac{1}{N-1}}\left(2\left(\frac{j(N-1)!}{2}\right)^{\frac{2}{N-1}}+(N-1)\left(\frac{j(N-1)!}{2}\right)^{\frac{1}{N-1}}+N-2\right)\\
%\sim 2\left(\frac{j(N-1)!}{2}\right)^{\frac{3}{N-1}}=2^{\frac{N-4}{N-1}}(N-1)!^{\frac{3}{N-1}}j^{\frac{3}{N-1}},
%\end{multline*}
%as $j\rightarrow +\infty$. 
We note that this is exactly \eqref{weyl_NBS}. Indeed, recalling that $|\partial B|=N\omega_N$, a standard computation  shows that 
$$
2^{\frac{N-4}{N-1}}(N-1)!^{\frac{3}{N-1}}=\frac{16\pi^3}{\omega_{N-1}^{\frac{3}{N-1}}}\cdot\frac{1}{|\partial\Omega|^{\frac{3}{N-1}}},
$$
for which it is useful to note that 
$$
\omega_N\omega_{N-1}=\frac{2^N\pi^{N-1}}{N!}.
$$
In the same way we verify that
\begin{equation*}%\label{asymptotic_eta_k}
\eta_j\sim 2^{\frac{N-2}{N-1}}(N-1)!^{\frac{1}{N-1}}j^{\frac{1}{N-1}},
\end{equation*}
\begin{equation*}%\label{asymptotic_LM_k}
\lambda_j(\mu)\sim \frac{3}{4}2^{\frac{N-2}{N-1}}(N-1)!^{\frac{1}{N-1}}j^{\frac{1}{N-1}},
\end{equation*}
and
\begin{equation*}%\label{asymptotic_ML_k}
\mu_j(\lambda)\sim \frac{3}{4}2^{\frac{N-4}{N-1}}(N-1)!^{\frac{3}{N-1}}j^{\frac{3}{N-1}},
\end{equation*}
as $j\rightarrow+\infty$, and these asymptotic formulas correspond to formulas \eqref{ball_eta}, \eqref{ball_LM} and \eqref{ball_ML}.

%\begin{comment}

%%%%%%%%%%%%%%%%%%%%%%%%%%%%%%%%%%%%%%%%%%%%%%%%%%%%%%%%%%%%%%%%%%%%%%%%%%%%%%%%%%%%%%%%%%%%%%%%%%%%%%%%%%%%%%%%%%%%%%%%%%%%%%%%%%%%%%%%%%%%%%%%%%%%%%%%%%%
%%%%%%%%%%%%%%%%%%%%%%%%%%%%%%%%%%%%%%%%%%%%%%%%%%%%%%%%%%%%%%%%%%%%%%%%%%%%%%%%%%%%%%%%%%%%%%%%%%%%%%%%%%%%%%%%%%%%%%%%%%%%%%%%%%%%%%%%%%%%%%%%%%%%%%%%%%%

%%%%%%%%%%%%%%%%%%%%%%%%%%%%%%%%%%%%%%             APPENDIX: BEHAVIOR OF THE EIGENVALUE ON THE OTHER SIDE                       %%%%%%%%%%%%%%%%%%%%%%%%%%%%%%%%%%%%

%%%%%%%%%%%%%%%%%%%%%%%%%%%%%%%%%%%%%%%%%%%%%%%%%%%%%%%%%%%%%%%%%%%%%%%%%%%%%%%%%%%%%%%%%%%%%%%%%%%%%%%%%%%%%%%%%%%%%%%%%%%%%%%%%%%%%%%%%%%%%%%%%%%%%%%%%%%
%%%%%%%%%%%%%%%%%%%%%%%%%%%%%%%%%%%%%%%%%%%%%%%%%%%%%%%%%%%%%%%%%%%%%%%%%%%%%%%%%%%%%%%%%%%%%%%%%%%%%%%%%%%%%%%%%%%%%%%%%%%%%%%%%%%%%%%%%%%%%%%%%%%%%%%%%%%

\section{The \eqref{BSM} and \eqref{BSL} problems for $\mu>0$ and $\lambda>\eta_1$}\label{appendixA}

In this section we briefly discuss problem \eqref{BSM} for $\mu>0$ and problem \eqref{BSL} for $\lambda>\eta_1$. 

We begin with problem \eqref{BSM}. Assume that $\mu\in\mathbb R$ is such that $\xi_j<\mu<\xi_{j+1}$ for some $j\in\mathbb N$. Recall that $\xi_j$ denote the eigenvalues of problem \eqref{NBS}. We denote by $U_j$ the subspace of $H^2(\Omega)$ generated by all eigenfunctions $u_i$ associated with the eigenvalues $\xi_i$ with $i\leq j$ and we set
$$
U_j^{\perp}=\left\{u\in H^2(\Omega):\mathcal Q_{\mu,D}(u,\varphi)=0\,,\ \ \ \forall\varphi\in U_j\right\}.
$$
The space $U_j^{\perp}$ is a closed subspace of $H^2(\Omega)$. We have the following result.

\begin{thm}\label{thm_mu_xi}
Let $\Omega$ be a bounded domain in $\mathbb R^N$ of class $C^{0,1}$ and assume that $\xi_j<\mu<\xi_{j+1}$ for some $j\in\mathbb N$. Then
\begin{equation}\label{split2}
H^2(\Omega)=U_j\oplus U_j^{\perp}.
\end{equation}
Moreover, there exists $b\geq 0$ such that the quadratic form $\mathcal Q_{\mu,D}(u,v)+b(\gamma_1(u),\gamma_1(v))_{\partial\Omega}$, $u,v\in H^2(\Omega)$, is coercive on $U_j^{\perp}$. 
\end{thm}

Note that the decomposition \eqref{split2} is not straightforward since $\mathcal Q_{\mu,D}$ does not define a scalar product for $\mu>0$. However, in order to prove Theorem \ref{thm_mu_xi} one can easily adapt the analogous proof in \cite[Theorem 5.1]{lamberti_stratis} but we omit the details.

We observe that, whenever $u\in U_j^{\perp}$, one can consider only test functions $\varphi\in U_j^{\perp}$ in the weak formulation of \eqref{weak_BSM}. Indeed, adding to $\varphi$ a test function $\tilde\varphi\in U_j$ leaves both sides of the equation unchanged. Hence one can perform the same analysis as in the case $\mu<0$ and state an analogous version of Theorem \ref{main-BSM} with the space $H^2(\Omega)$ replaced by $U_j^{\perp}$. We leave this to the reader.

\begin{rem}\label{C2}
If $\mu=\xi_j$ for some $j\in\mathbb N$ the situation is more involved and is not analyzed here. %We note that the case $\mu=\xi_1=0$ has been treated separately in Subsection \ref{neu_prob_2}. 
%This is a special case which, in principle, may occur also with other eigenvalues $\xi_j$. 
However, if we assume that $\xi_j$ is an eigenvalue of \eqref{weak_NBS} of multiplicity $m$ such that
\begin{equation}\label{quot_p}
\mathcal Q_{\sigma}(u,\varphi)=\xi_j(\gamma_0(u),\gamma_0(\varphi))_{\partial\Omega}\,,\ \ \ \forall\varphi\in H^2(\Omega), u\in U_{\xi_j},
\end{equation}
where $U_{\xi_j}$ is the eigenspace generated by all the eigenfunctions $\{u_j^1,...,u_j^m\}$ in $\mathcal H^2_{0,N}(\Omega)$ associated with $\xi_j$, the problem becomes simpler. Indeed, any function $w\in H^2(\Omega)$ can be written in the form $w=u+v$, where $u\in U_{\xi_j}$ and $v\in H^2_{\xi_j}(\Omega)$, where
$$
H^2_{\xi_j}(\Omega):=\left\{v\in H^2(\Omega): (\gamma_0(v),\gamma_0(u_j^i))_{\partial\Omega}=0{\rm\ for\ all\ }i=1,...,m\right\}.
$$ 
Hence, whenever $u\in H^2_{\xi_j}(\Omega)$, one can consider only test functions in $\varphi\in H^2_{\xi_j}(\Omega)$ in the weak formulation \eqref{weak_BSM} with $\mu=\xi_j$. In fact, adding to $\varphi$ a function $\tilde\varphi\in U_{\xi_j}$ leaves both sides of the equation unchanged. Thus we can perform the same analysis of Theorem \ref{thm_mu_xi} with $H^2(\Omega)$ replaced by $H^2_{\xi_j}(\Omega)$.

Note that equation \eqref{quot_p} is satisfied with $\xi_1=0$ and $u$ a constant function.
\end{rem}

We also have an analogous result for problem \eqref{BSL} with $\lambda>\eta_1$. Assume that $\lambda\in\mathbb R$ is such that $\eta_j<\lambda<\eta_{j+1}$ for some $j\in\mathbb N$. Recall that $\eta_j$ denote the eigenvalues of problem \eqref{DBS}. We denote by $V_j$ the subspace of $H^2(\Omega)$ generated by all eigenfunctions $v_i$ associated with the eigenvalues $\eta_i$  with $i\leq j$ and we set
$$
V_j^{\perp}=\left\{v\in H^2(\Omega):\mathcal Q_{\lambda,N}(v,\varphi)=0\,,\ \ \ \forall\varphi\in V_j\right\}.
$$
The space $V_j^{\perp}$ is a closed subspace of $H^2(\Omega)$. We have the following result.

\begin{thm}\label{thm_lambda_eta}
Let $\Omega$ be a bounded domain in $\mathbb R^N$ of class $C^{0,1}$ and assume that $\eta_j<\lambda<\eta_{j+1}$. Then
\begin{equation}\label{split3}
H^2(\Omega)=V_j\oplus V_j^{\perp}.
\end{equation}
Moreover, there exists $b\geq 0$ such that the quadratic form $\mathcal Q_{\lambda,N}(u,v)+b(\gamma_0(u),\gamma_0(v))_{\partial\Omega}$, $u,v\in H^2(\Omega)$, is coercive on $V_j^{\perp}$. 
\end{thm}

We observe that, whenever $v\in V_j^{\perp}$, one can consider only test functions $\varphi\in V_j^{\perp}$ in the weak formulation of \eqref{weak_BSL}. In fact, adding to $\varphi$ a test function $\tilde\varphi\in V_j$ leaves both sides of the equation unchanged. Hence one can perform the same analysis as in the case $\lambda<\eta_1$ and state an analogous version of Theorem \ref{main-BSL} with the space $H^2(\Omega)$ replaced by $V_j^{\perp}$. We leave this to the reader.

\begin{rem}
As in Remark \ref{C2}, one can treat the case $\lambda=\eta_j$ for some $j$ in the special situation when $\eta_j$ is an eigenvalue of \eqref{weak_DBS} of multiplicity $m$ such that
$$
\mathcal Q_{\sigma}(v,\varphi)=\eta_j(\gamma_1(v),\gamma_1(\varphi))\,,\ \ \ \forall\varphi\in H^2(\Omega), v\in V_{\eta_j},
$$
where $V_{\eta_j}$ is the eigenspace generated by all the eigenfunctions $\{v_j^1,...,v_j^m\}$ in $\mathcal H^2_{0,D}(\Omega)$ associated with $\eta_j$. This happens in the case of the unit ball  with $\eta_1=1$. We observe that any function $w\in H^2(\Omega)$ can be written in the form $w=v+u$, where $v\in V_{\eta_j}$ and $u\in H^2_{\eta_j}(\Omega)$, where
$$
H^2_{\eta_j}(\Omega):=\left\{u\in H^2(\Omega): (\gamma_1(u),\gamma_1(v_j^i))_{\partial\Omega}=0{\rm\ for\ all\ }i=1,...,m\right\}.
$$ 
Hence, whenever $v\in H^2_{\eta_j}(\Omega)$, one can consider in the weak formulation \eqref{weak_BSL} with $\lambda=\eta_j$ only test functions in $\varphi\in H^2_{\eta_j}(\Omega)$. In fact, adding to $\varphi$ a function $\tilde\varphi\in V_{\eta_j}$ leaves both sides of the equation unchanged. Thus we can perform the same analysis of Theorem \ref{thm_lambda_eta} with $H^2(\Omega)$ replaced by $H^2_{\eta_j}(\Omega)$.
\end{rem}

We conclude this section with a few more remarks. We have observed that the eigenvalue $\mu=0$ is always an eigenvalue of \eqref{weak_BSL} when $\lambda<\eta_1$, and  that the constant functions belong to the eigenspace associated with $\mu=0$, and in particular belong to the space $\mathcal H^2_{0,N}(\Omega)$ and are eigenfunctions associated with the first eigenvalue $\xi_1=0$ of problem \eqref{NBS}. Such a situation may occur also for other eigenvalues, as well as for problem \eqref{weak_BSM} (as we have seen in the case of the ball in $\mathbb R^N$ with the eigenvalue $\lambda(\mu)=\eta_1=1$ for all $\mu\in\mathbb R$). The following lemma clarifies this phenomenon.

\begin{lem}\label{degenerate-M}
Let $j\in\mathbb N$. Then one of the following two alternatives occur for problem \eqref{weak_BSM}:
\begin{enumerate}[i)]
\item $\lambda_j(\mu)<\eta_j$ for all $\mu\in(-\infty,0)$;
\item there exists $\mu_0\in(-\infty,0)$ such that $\lambda_j(\mu_0)=\eta_j$. In this case, $\eta_j=\lambda_j(\mu_0)$ for all $\mu\in(-\infty,\mu_0]$ and $\eta_j$ is an eigenvalue of problem \eqref{weak_BSM} for any  $\mu\in\mathbb R$.
\end{enumerate}
\end{lem}
\begin{proof}
From the Min-Max Principles \eqref{minmax_DBS} and \eqref{minmaxM} we have that $\lambda_j(\mu)\leq\eta_j$ for all $\mu\in(-\infty,0)$, $j\in\mathbb N$. Assume now that there exists $\mu_0\in(-\infty,0)$ such that $\lambda_j(\mu_0)=\eta_j$. Again from \eqref{minmax_DBS} and \eqref{minmaxM} we deduce that we can choose an eigenfunction $v_{j,\mu_0}$ associated with $\lambda_j(\mu_0)$ which belong to $\mathcal{H}^2_{0,D}(\Omega)$ and which coincide with an eigenfunction $u_j$ of problem \eqref{weak_DBS} associated with $\eta_j$. In fact the eigenfunctions associated with $\lambda_j(\mu)$ are exactly the functions realizing the equality in \eqref{minmaxM}.
We also note that
$$
\mathcal Q_{\sigma}(u_j,\varphi)=\mathcal Q_{\mu,D}(u_j,\varphi)\,,\ \ \ \forall\varphi\in H^2(\Omega),\mu\in\mathbb R,
$$
hence 
$$
\mathcal Q_{\mu,D}(u_j,\varphi)=\eta_j(u_j,\varphi)_{\partial\Omega}\,,\ \ \ \forall\varphi\in H^2(\Omega),\mu\in\mathbb R,
$$
hence $\eta_j$ is an eigenvalue of \eqref{weak_BSM} for all $\mu\in\mathbb R$ and in particular $\eta_j=\lambda_j(\mu)$ for all $\mu\in(-\infty,\mu_0]$. This concludes the proof.
\end{proof}

In the same way one can prove the following.

\begin{lem}\label{degenerate-L}
Let $j\in\mathbb N$. Then one of the following two alternatives occur for problem \eqref{weak_BSL}:
\begin{enumerate}[i)]
\item $\mu_j(\lambda)<\xi_j$ for all $\lambda\in(-\infty,\eta_1)$;
\item there exists $\lambda_0\in(-\infty,\eta_1)$ such that $\mu_j(\lambda_0)=\xi_j$. In this case, $\xi_j=\mu_j(\lambda_0)$ for all $\lambda\in(-\infty,\lambda_0]$ and $\xi_j$ is an eigenvalue of problem \eqref{weak_BSL}, for any  $\lambda\in\mathbb R$.
\end{enumerate}
\end{lem}

%We remark that the situation described in Lemmas \ref{degenerate-M} and \ref{degenerate-L} can be observed in the case that $\Omega$ is a ball. In particular $\eta_1=1$ is an eigenvalue of \eqref{weak_BSM} for all $\mu\in\mathbb R$ and, trivially, $\xi_1=0$ is an eigenvalue of \eqref{weak_BSL} for all $\lambda\in\mathbb R$ (this second fact is true for all $\Omega\in\mathbb R^N$). Note that case $ii)$ of Lemma \ref{degenerate-L} always happens with $\mu_1(\lambda)=\xi_1=0$.

%%%%%%%%%%%%%%%%%%%%%%%%%%%%%%%%%%%%%%%%%%%%%%%%%%%%%%%5 DA RIMUOVERE UNA VOLTA TOLTO IL COMMENTO %%%%%%%%%%%%%%%%%%%%%%%%%%%%%%%%%%%%%%%%%%%%%%%%

%\appendix
%\section{Behavior of the eigenvalues of \eqref{BSM} and \eqref{BSL} for $\lambda\geq\eta_1$ and $\mu\geq 0$}\label{appendixA}

%\bibliography{bibliography}{}
%\bibliographystyle{abbrv}
\def\cprime{$'$} \def\cprime{$'$} \def\cprime{$'$} \def\cprime{$'$}
  \def\cprime{$'$}

% ----------------------------------------------------------------

\end{document}